\documentclass[25mm]{gSTA2e}
\usepackage{array}
\usepackage{color}
\usepackage{amsthm}
\usepackage[utf8]{inputenc}
\usepackage{dsfont}
\usepackage{amsmath}
\usepackage{amssymb}
\usepackage{chngcntr}
\usepackage{natbib}
\usepackage{enumerate}
\renewcommand{\baselinestretch}{1.3}

\makeatletter
\@addtoreset{equation}{section}

\makeatother

\newcounter{Fig}[figure]

\newcounter{Tab}[table]

  {\refstepcounter{Tab}
   \addtocounter{table}{1}
   \samepage\vspace{0.2cm}
   \centerline{#1} \nobreak
   \begin{verse}{Table \thesection.\arabic{Tab}:}
  }%
  {\end{verse} \vspace{0.2cm}}

\relax

\newcommand{\bqa}{\begin{eqnarray*}}
\newcommand{\eqa}{\end{eqnarray*}}
\newcommand{\bqan}{\begin{eqnarray}}
\newcommand{\eqan}{\end{eqnarray}}
\newcommand{\bqt}{\begin{quote}}
\newcommand{\eqt}{\end{quote}}
\newcommand{\bt}{\begin{tabbing}}
\newcommand{\et}{\end{tabbing}}
\newcommand{\bit}{\begin{itemize}}
\newcommand{\eit}{\end{itemize}}
\newcommand{\ben}{\begin{enumerate}}
\newcommand{\een}{\end{enumerate}}
\newcommand{\beq}{\begin{equation}}
\newcommand{\eeq}{\end{equation}}
\newcommand{\bdefi}{\begin{definition}}
\newcommand{\edefi}{\end{definition}}
\newcommand{\bpro}{\begin{proposition}}
\newcommand{\epro}{\end{proposition}}
\newcommand{\blem}{\begin{lemma}}
\newcommand{\elem}{\end{lemma}}
\newcommand{\bth}{\begin{theorem}}
\newcommand{\bco}{\begin{corollary}}
\newcommand{\eco}{\end{corollary}}
\newcommand{\bdes}{\begin{description}}
\newcommand{\edes}{\end{description}}
\newcommand{\bre}{\begin{remark}}
\newcommand{\ere}{\end{remark}}

\newcounter{compteur}

 \newtheorem{definition}{compteur}[section]
\newtheorem{proposition}[compteur]{Proposition}
 \newtheorem{lemma}[compteur]{Lemma}
 \newtheorem{theorem}[compteur]{Theorem}
 \newtheorem{corollary}[compteur]{Corollary}

\newcommand{\indep}{\bot\!\!\!\bot}

\usepackage[T1]{fontenc}
\newcounter{assump}
\newtheorem{assum}[assump]{Assumption}

\newtheorem{property}{Property}

\DeclareMathOperator*{\argmin}{arg\,min}
\begin{document}
\renewcommand{\baselinestretch}{1.2}
%\lhead[\fancyplain{} \leftmark]{}
%\chead[]{}
%\rhead[]{\fancyplain{}\rightmark}
%\cfoot{}
%\headrulewidth=0pt

\title{{\itshape Semiparametric inference for the recurrent events process by means of a single-index model}}
\author{Olivier Bouaziz$^{a}$$^{\ast}$\thanks{Corresponding author$^\ast$. Email: olivier.bouaziz@parisdescartes.fr \vspace{6pt}}, Ségolen Geffray$^{b}$ and Olivier Lopez$^{c,d}$\\\vspace{6pt} $^{a}${\em{MAP5, Universit\'e Paris Descartes, 45 rue des Saints P\`eres, 75270 Paris Cedex 06}}; $^{b}${\em{IRMA, Universit\'e de Strasbourg, 7 rue Ren\'e-Descartes 67084 Strasbourg Cedex}}; $^{c}${\em{ENSAE Paris-Tech \& CREST, Laboratoire de Finance et d'Assurance, 3 avenue Pierre Larousse, 92245 Malakoff Cedex France}}; $^{d}$ {\em{Sorbonne Universités, UPMC Université Paris VI, EA 3124, LSTA, 4 place Jussieu 75005 Paris}}\\}
\maketitle

%{\em{ENSAE Paris-Tech \& CREST, Laboratoire de Finance et d'Assurance, 3 avenue Pierre Larousse, 92245 Malakoff Cedex France};} $^{d}${Sorbonne Universités, UPMC Université Paris VI, EA 3124, LSTA, 4 place Jussieu 75005 Paris}\\
%\maketitle

\begin{abstract}
In this paper, we introduce new parametric and semiparametric regression techniques for a recurrent event process subject to random right censoring. We develop models for the  cumulative mean function and provide asymptotically normal estimators.
Our semiparametric model which relies on a single-index assumption can be seen as a dimension reduction technique that, contrary to a fully nonparametric approach, is not stroke by the curse of dimensionality when the number of covariates is high. 
%We obtain the asymptotic normality of both the parametric and nonparametric estimators.
We discuss data-driven techniques to choose the parameters involved in the estimation procedures and provide a simulation study to support our theoretical results.
\end{abstract}

\vspace{9pt}
\par

{\small
\parindent 0cm
%$ ^1$ MAP5, Universit\'e Paris Descartes, 45 rue des Saints P\`eres, 75270 Pari%s Cedex 06, France, E-mail: olivier.bouaziz@parisdescartes.fr\\
% $^2$ IRMA, UMR 7501, 7 rue René-Descartes, 67084 Strasbourg Cedex, France,\\ 
% E-mail: geffray@math.unistra.fr\\
% $^3$ Laboratoire de Statistique Th\'eorique et Appliqu\'ee, Universit\'e Paris VI, 175 rue du Chevaleret, 75013 Paris, France, E-mail: olivier.lopez0@upmc.fr\\
This work is supported by French Agence Nationale de la Recherche (ANR) ANR Grant ``Prognostic'' ANR-09-JCJC-0101-01.
}\\

\fontsize{10.95}{14pt plus.8pt minus .6pt}\selectfont
%\counterwithout{Lemma}{section}
\counterwithout{equation}{section}
\counterwithin{equation}{chapter}
\setcounter{chapter}{1}
%-1
%\setcounter{section}{0}
%\setcounter{equation}{0} 

\noindent {\bf 1. Introduction}

The modeling of recurrent events has become a crucial issue in various application fields of statistical inference such as clinical and epidemiological studies, insurance or actuarial science, in particular in the presence of a terminal event.
Among many examples, one can mention the modeling of asthma, of epileptic seizures in the presence of death or of repeated warranty claims with possibility of contract breaking. 
In these settings, we consider the problem of predicting or identifying the causes which influence the number of such events occurring during a given time period. 
A natural way to measure the impact of covariates on the recurrence of these events consists of estimating the cumulative function of some recurrent event process of interest conditionally on covariates. 
In this paper, our aim consists of developing both parametric and semiparametric inference for regression models suited to the conditional cumulative mean function in the presence of independent right-censoring and terminal event. 

Let $\widetilde N(t)$ represent the number of recurrent events occurring in the time interval $[0,t]$ for $t\geq 0$.
In the literature, various regression models tailored for the recurrent event process $\widetilde N(\cdot)$ have been considered. 
These regression models can be roughly divided into two categories which are Cox-type regression models and accelerated-failure-time (AFT)-type regression models.
Cox-type regression models (see eg Andersen and Gill \citeyearpar{AndersenGill1982} or Sun and Wei \citeyearpar{Sun-Wei-2000})
typically assume that a vector of covariates $Z$ acts on the conditional cumulative mean function through a multiplicative effect on an unspecified baseline function $\mu_0(\cdot)$. 
For instance, Sun and Wei \citeyearpar{Sun-Wei-2000} assume that:
\begin{align}\label{E:Cox}
E[\widetilde N(t)|Z]=\exp(\beta' Z)\mu_0(t) 
\end{align}
where $\beta$ is an unknown vector of parameters.
AFT-type regression models (see eg Lin et al. \citeyearpar{Lin-etal1998}) typically
assume that a vector of covariates $Z$ acts on the conditional cumulative mean function through a time-scaling effect on an unspecified baseline function $\mu_0(\cdot)$:
\begin{align}\label{E:AFT}
E[\widetilde N(t)|Z]=\mu_0(t\exp(\beta' Z))
\end{align}
even though variations of AFT-like regression models exist, for example, in Ghosh \citeyearpar{Ghosh2004} under the form
\begin{align}\label{E:AFT2}
E[\widetilde N(t)|Z]=\exp(-\beta' Z)\,\mu_0(t\exp(\beta' Z)).
\end{align}
However, none of the aforementioned papers takes into account terminal event occurrence or dependent censoring.
One of the most developed inferential approaches in the presence of both dependent and independent censoring can be found in Ghosh and Lin \citeyearpar{GhoshLin2003} where the recurrent event process $\widetilde N(\cdot)$ is modeled by means of AFT-type regression models. 
Frailty and time-dependent covariates extensions of Cox-type regression models have also been considered in Huang and Wang \citeyearpar{Huang-Wang-2004} and Huang et al. \citeyearpar{Huang-etal2010}.

Technically speaking, the main advantage of these kinds of models stands in the simplicity of the regression function. But they unfortunately face the disadvantage (with respect to a purely nonparametric approach) of relying on strong modeling assumptions that may not hold in practice.
However, it turns out that, while allowing full flexibility, the nonparametric approach is known to fail when the number of covariates is high (greater than $3$ in practice) which is the so-called ``curse-of-dimensionality''.
On the other hand, single-index models (see e.g. Ichimura~\citeyearpar{Ichimura93}, H\"ardle et al. \citeyearpar{Hardle1993}, Xia et al. \citeyearpar{xia2002} and Bouaziz and Lopez \citeyearpar{BouazizLopez})  aim to achieve a compromise between a parametric approach and a nonparametric one.  
The basic idea behind this class of models is to assume that the regression function depends on an unknown linear combination of the covariates. 
In our framework, this corresponds to assuming that $E[\widetilde N(t)|Z]=E[\widetilde N(t)|\theta_0'Z]$, for some unknown vector $\theta_0$. 
Hence if this parameter was known, estimating $E[\widetilde N(t)|Z]$ would reduce to a nonparametric problem in dimension one.
This dimension reduction assumption allows to achieve better convergence rates and still ensures enough flexibility to be adapted to a large number of practical cases. 
Moreover, these models can also generalize standard models such as the Cox or AFT regression models.

In this work, we first study a general parametric regression model for a recurrent event process in the presence of independent right-censoring and terminal event. 
We then study a semiparametric generalization which relies on a single-index assumption. 
We propose new procedures to estimate both the index and the conditional cumulative mean regression function and provide a detailed asymptotic study of the proposed estimators.
Compared to uncensored single-index models adapted to mean-regression, see e.g. Ichimura~\citeyearpar{Ichimura93}, the presence of censoring usually deteriorates the quality of estimation in the tail of the distribution. 
Therefore, in our specific setting of regression for recurrent events with censoring and terminal event, we introduce a weight function designed to compensate for the lack of information induced by independent censoring.
The main novelty of our procedure stands in the fact that this weight function may be chosen using data-driven techniques. 
We then discuss a data-driven way of calibrating the parameters involved in the estimation procedures.

The paper is organized as follows. In Section 2\ref{secmodel}, we define the parametric and semiparametric models, introduce the corresponding estimators and explain the general methodology. Asymptotic results for our new estimators are presented in Section 3\ref{secasympt}. Simulation studies are carried out in Section 4\ref{secsimu} to investigate on the performance of our methods for moderate sample size. Technical results are postponed to the Appendix in Section 6\ref{secappend}.\\

\par

\noindent {\bf 2. Model assumptions and methodology} \label{secmodel}
%-1

In this section, we present the regression framework. Specifically, Section 2.1\ref{secsecreg} introduces the different regression models considered in this paper. 
Section 2.2\ref{secsecest} presents the estimation procedures. 
They are based on a least-squares type criterion involving a rescaled process and a weighting measure introduced in Section 2.2.1\ref{secsecsecresc} which enable to correct the impact of censoring. 
The parametric case is dealt with in Subsection 2.2.2\ref{secsecsecpara} and the semiparametric case is treated in Subsection 2.2.3\ref{secsecsecsemi}.\\

\noindent {\bf 2.1 Regression models for the recurrent event process} \label{secsecreg}
\setcounter{chapter}{2}
\setcounter{equation}{0} 
%\addtocounter {equation} {1}

Consider the process $\widetilde N(\cdot)$ where we recall that $\widetilde N(t)$ stands for the number of recurrent events occurring in the time interval $[0,t]$ for any $t\geq 0$. 
Let $D$ denote the random variable representing the time until occurrence of a terminal event. In clinical applications, this variable $D$ may stand for the death time of a patient. 
For insurance applications, $D$ can represent the warranty length (which can be random if the client has the possibility of breaking the contract) or the lifetime of the insured good.  
Then introduce the recurrent event process $N^*(\cdot)=\widetilde N(\cdot \wedge D)$ (where $a\wedge b$ denotes the infimum between $a$ and $b$) which can be seen as a piecewise constant function with jumps only on $[0,D]$.
In this paper, we aim to infer on the cumulative conditional mean function given for $t\geq 0$ by
$$\mu(t|z)=E\left[N^*(t)|Z=z\right]$$ 
where $Z$ is a $d-$dimensional vector of covariates.
Note that  our interest is focused on the process $N^*(\cdot)$ and not on $\widetilde N(\cdot)$. 
The role of the variable $D$ is to stop the process $\tilde N$ but, on the opposite of Ghosh and Lin \citeyearpar{GhoshLin2003} for instance, we are not interested in modeling its distribution. 
Note also that no assumptions regarding the dependence between $N^*$ and $D$ are made. 
Our setting is similar to the one of Ghosh and Lin \citeyearpar{GhoshLin2000},  Dauxois and Sencey \citeyearpar{DauxoisSencey09} or Bouaziz et al. \citeyearpar{Bouazizetal13}. 

We now present the two different models for $\mu$ that are studied throughout this paper.\\

{\bf Model 1 : parametric case.}
\begin{equation}
\label{modelpara} \mu(t|z) = \mu_0(t,z;\theta_0),
\end{equation}
where $\theta_0$ is unknown in some parameter space $\Theta\subset \mathbb{R}^{d'}$ (where $d'$ may be different from $d$) and $\mu_0$ is a known function.\\

{\bf Model 2 : semiparametric case.}
\begin{equation}
\label{modelsemipara} \mu(t|z) = \mu_{\theta_0}(t,\theta_0'z),
\end{equation}
where $\theta_0$ is unknown in some parameter space $\Theta\subset \mathbb{R}^d$, where \\$\mu_{\theta}(t,u)=E[N^*(t)|\theta 'Z=u]$ and where the family
of functions $\mathcal{F}=\{\mu_{\theta}:\theta\in \Theta\}$ is unknown. 
We impose that the first component of $\theta_0$ is 1 to identify this parameter. 
Indeed, if we do not make any assumption on $\theta_0$, this parameter would only be defined up to a multiplicative constant. Note that another equivalent condition would consist of imposing that $\theta_0$ is of norm 1 for any given norm on $\mathbb{R}^d.$
Model 2 is a single-index model (see e.g. H\"ardle et al. \citeyearpar{Hardle1993}), since it consists of assuming that $\mu(t|z)$ depends on the covariates only through a linear combination $\theta_0'z.$ 
In the notation that we introduce in (\ref{modelsemipara}) and that we will use throughout this paper, we emphasize the fact that the function $\mu$ depends on $\theta$ in two different ways. 
The notation $\mu_{\theta}$ indicates that, for two different values of $\theta,$ the conditional distribution that one is considering is not the same. 
Indeed, the distribution of $N^*(\cdot)$ conditionally on $\theta_1'Z$ is, in general, not the same as the distribution of $N^*(\cdot)$ conditionally to $\theta'_2Z$ for two different vectors $\theta_1$ and $\theta_2.$ 
Moreover, this function $\mu_{\theta_0}(t,u)$ is evaluated at the point $\theta_0'Z.$
This distinction is essential to obtain some crucial properties of single-index models, and explains why the vector of partial derivatives with respect to $\theta$ has a relatively complex form (see Lemma 5 in Supplementary material). %\ref{lemmacalculgradient}
To ensure consistency of single-index approaches, a continuity property of the map $\theta\rightarrow \mu_{\theta}(\cdot,\cdot)$ (which is a map from $\Theta$ towards some space of functions) will be required.
 To illustrate the notations, consider the particular case where
\begin{equation*}E[N^*(t)|Z]=(\theta_0'Z+5)\mu_0(t)\end{equation*}
for some function $\mu_0(\cdot)$.
A quick computation shows that 
\begin{align*}
E\left[N^*(t)|\theta'Z\right] &= E\left[E[N^*(t)|Z]|\theta'Z\right] \\
&= (\theta_0'E[Z|\theta'Z]+5)\mu_0(t),
\end{align*}
which gives $$\mu_{\theta}(t,u)=(\theta'_0E[Z|\theta'Z=u]+5)\mu_0(t).$$
\\

The appealing feature of Model 1 stands in the simplicity of the regression function.
However, like every parametric procedure, it relies on strong assumptions which have few chances to hold in practice.
On the opposite, a fully nonparametric procedure requires fewer assumptions but suffers from the so-called ``curse of dimensionality'' when the number of covariates is high. Therefore, Model 2 appears as a good compromise between the parametric approach and the nonparametric one. Indeed it is more flexible than a fully parametric one but is not stroke by the curse of dimensionality since it relies on a dimension reduction assumption.
Moreover, Model 2 can be seen as a generalization of the models exposed in equations (\ref{E:Cox}) to (\ref{E:AFT2}). \\

One does not generally observe $N^*(\cdot)$ on the whole time interval $[0,D]$ because the random variable $D$ is subject to right-censoring. Let $C$ be a positive random variable standing for the censoring time. The observation time $T$ is then given by $T=D\wedge C$. Hence, instead of observing $N^*(t)$ for $t\in [0,D]$, one only observes $N(t)=N^*(t\wedge C)$ for $t\in [0,D]$. Letting $\delta=I(D\leq C)$, the observations consist of
$n$ i.i.d. replications $(T_i,\delta_i,Z_i,N_i(\cdot))_{1\leq i\leq n}$ of $(T,\delta,Z,N(\cdot))$.
Let us introduce the cumulative distribution functions of the observed variables in the censored data model:
\[\begin{cases}
H(t) = P(T\leq t),\\
F(t) = P(D\leq t),\\
G(t) = P(C\leq t).
\end{cases}\] 
We also define $\tau_H=\inf\{t:H(t)=1\}$ the right endpoint of the support of the random variable $T$.
In the sequel, we need the two following assumptions to identify these distribution functions.

\begin{assum}\label{sym}
For a counting process $L$ let us denote $dL(t)=L(t)-L(t-)$ (where $L(t-)=\lim_{u\rightarrow t,u<t} L(u)$) the jump of process $L$ at time $t$. Assume that
\[\begin{cases}
P\big(dN^*(C)\neq 0\big)=0,\\
P(D=C)=0.
\end{cases}\]
\end{assum}
This is a common assumption in the context of recurrent events which prevents us from ties between the occurrence times of death, censoring and recurrent events. 

\begin{assum}\label{symrec}
We write down $A\indep B$ when two random variables $A$ and $B$ are independent. Assume that
\[
\begin{cases}
C \indep (N^*(\cdot),D),\\
P(C\leq t|N^*(\cdot),Z,D)=P(C\leq t|N^*(\cdot),D) \text{ for } t\in [0,\tau_H].\end{cases}
\]
\end{assum}
Assumption 2 holds in the particular case where $C$ is independent of \\$(N^*(\cdot),D,Z)$ but is slightly more general since it does not require the independence between $C$ and $Z.$ 
Similar assumptions are often considered in the literature on the Kaplan-Meier estimator for the survival distribution function with covariates. 
To study the Kaplan Meier estimator, Stute~\citeyearpar{Stute93} assumed that $P(C\leq t|Z,D)=P(C\leq t|D)$ and $C \indep D$ (see the discussion in Stute~\citeyearpar{Stute93} for this assumption). 
Our Assumption \ref{symrec} is the natural extension of the assumption proposed by Stute~\citeyearpar{Stute93}, but now in the presence of a recurrent event process $N^*$.

Alternatively, one could assume that $C\indep (N^*(\cdot),D)$ conditionally on $Z.$ 
Using such an assumption instead of Assumption \ref{symrec} would require to modify the approach described below, by replacing the Kaplan-Meier estimator of the distribution of $C$ by a conditional Kaplan-Meier estimator as the one proposed by Beran \citeyearpar{Beran} and studied by Dabrowska \citeyearpar{Dabrowska}. 
In the following, we do not focus on the theoretical behaviour of this modification, which is left to future research.\\
\par

\noindent {\bf 2.2 Estimation procedure} \label{secsecest}\\
\noindent {\bf 2.2.1 Heuristics for the rescaled process and the weighting measure} \label{secsecsecresc}

Our objective is to estimate $E[N^*(t)|Z]$ successively under Model 1 and Model 2 from the i.i.d. sample $(T_i,\delta_i,Z_i,N_i(\cdot))_{1\leq i\leq n}$.
In our regression framework, going back to the definition of the conditional expectation, it is quite natural to perform estimation of $E[N^*(t)|Z]$ using minimization of a least-squares-type criterion both in Model 1 and Model 2. 
With this method in mind, consider a least-squares criterion which is integrated over $[0,\tau_H)$ to control the trajectory of the process of interest over this time interval. This gives the following criterion, say under Model 1,
$$ \int_0^{\tau_H} E\left[\left(\mu_0(t,Z;\theta)-N^*(t) \right)^2 \right]dt$$
which is to be minimized with respect to $\theta$.
One of the difficulties we face when estimating the conditional expectation of $N^*(t)$ is that the process $N^*(\cdot)$ is not directly observed because of censoring. 
Hence, empirical versions of criteria like the above one can not be computed since they rely on $N^*(\cdot)$.  
To circumvent this difficulty we introduce a rescaled process $Y(\cdot)$ which is designed to compensate the censoring effects.
We define for any $t$ in $[0,\tau_H)$
\begin{equation} \label{rescale} 
Y(t) = \int_0^{t} \frac{dN(s)}{1-G(s-)}.
\end{equation}
The logic behind this rescaled process is similar to the approach used by Leurgans \citeyearpar{Leurgans-1987} in a censored regression framework.
In the definition (\ref{rescale}), the denominator is decreasing when $s$ grows to infinity. 
This means that we give more weight to the events we observe when $s$ is large and compensate for the lack of observations due to censoring for large $s$.
To go further in the definition of our least-squares criterion, notice that under Assumptions \ref{sym} and \ref{symrec}, we have for any $s$ in $[0,\tau_H)$
\begin{align}\label{NNstar}
E[dN(s)|Z]&=E[dN^*(s\wedge C)|Z]\nonumber \\
								&=E[dN^*(s)I(s\leq C)|Z]\nonumber\\
								&=E[dN^*(s)|Z](1-G(s-))
\end{align}
so that 
$$E[Y(t)|Z]=E[N^*(t)|Z].$$
The consequence is that we can now consider a modified least-squares criterion based on the estimated rescaled process $Y(\cdot)$, say again under Model 1,
$$ \int_0^{\tau_H} E\left[\left(\mu_0(t,Z;\theta)-Y(t) \right)^2 \right] dt,$$
that is the integrated squared error (see for example Bowman \citeyearpar{Bowman1984} in the context of cross validation for kernel estimators).
The other difficulty we have to face is that we have to ensure the finiteness of our least-squares criterion which is not guaranteed with the above definition. To circumvent this other difficulty, we will use a weighting measure $w$ specifically designed to ensure the finiteness of our criterion so that our criterion to be minimized will be of the form
$$ \int_0^{\tau_H} E\left[\left(\mu_0(t,Z;\theta)-Y(t) \right)^2 \right] dw(t).$$
We now consider in details separately the parametric case and the semiparametric case in the two subsections to come.\\

\noindent {\bf 2.2.2 The parametric case} \label{secsecsecpara}

Suppose that Model 1 is satisfied.
Let $w$ denote a measure such that\\$w\big([0,\infty)\big)<\infty$ and such that the quantity
\[M_w(\theta,\mu_0)= \int_0^{\tau_H} E\big[\mu_0(t,Z;\theta)^2\big]dw(t)-2\int_0^{\tau_H} E\big[Y(t)\mu_0(t,Z;\theta)\big]dw(t)\]
is finite. 
Let us notice that $\mu_0(\cdot,\cdot;\theta)$ and $Y(\cdot)$ may tend to infinity when $t\to \tau_H$. 
This remark leads us to introduce a supplementary condition on $w$ distinct from the fact that this measure has a finite total mass.
The true parameter value $\theta_0$ satisfies
\begin{equation}
\label{idealpara}
\theta_0=\argmin_{\theta \in \Theta} M_w(\theta,\mu_0)
\end{equation}
regardless of the choice of $w$.
To estimate $\theta_0$, it is natural to replace the function $M_w$ by an empirical version.
However, the rescaled process $Y(\cdot)$ can not be computed in practice since it relies on the distribution function $G$ which is usually unknown. To circumvent this other difficulty, we introduce an empirical counterpart of $Y(\cdot)$.
Let $T_{(n)}$ denote the last order statistic of the sample $(T_i)_{i=1,...,n}$. The distribution function $G$ can be consistently estimated on $[0,T_{(n)}]$ by the Kaplan-Meier estimator of $G$ denoted by $\hat{G}$ and given for $t$ in $[0,T_{(n)}]$ by
\[
\hat{G}(t) = 1-\prod_{i:T_i\leq t}\left(1-\frac{1}{\sum_{j=1}^nI(T_j\geq T_i)}\right)^{1-\delta_i}.
\]
Consequently, the process $Y(\cdot)$ can itself be estimated for $t$ in $[0,T_{(n)}]$ by
\begin{equation} \label{rescaleest} 
\hat{Y}(t) = \int_0^{t} \frac{dN(s)}{1-\hat{G}(s-)}.
\end{equation}
The empirical version of $M_w$ considered here is then
\[M_{n,w}(\theta,\mu_0)= \frac{1}{n}\sum_{i=1}^n\int_0^{T_{(n)}} \mu_0(t,Z_i;\theta)^2dw(t)-\frac{2}{n}\sum_{i=1}^n\int_0^{T_{(n)}} \hat{Y}_i(t)\mu_0(t,Z;\theta)dw(t).\]
This allows us to define an estimator of $\theta_0$ as
\begin{equation}\label{estimpara}
\hat{\theta}(w)=\argmin_{\theta \in \Theta}M_{n,w}(\theta,\mu_0).
\end{equation}
In the above definition, we emphasize the fact that this estimator depends on the choice of the measure $w.$
This measure $w$ is an important feature of our procedure. First, in some situations, the statistician may
wish to give more weight to some time intervals which are of higher importance. Moreover, the measure
$w$ is also useful to control the rescaled process. Indeed, in Equation (\ref{rescaleest}), the denominator 
goes to zero when $s$ grows large and $w$ can be precisely designed to avoid the practical problems caused by
these too small denominators. Therefore, the finite sample behaviour of our estimation procedure
strongly relies on a wise choice of the measure $w.$

The asymptotic results derived in Section 3\ref{secasympt} allow us to obtain asymptotic representations of $\hat{\theta}(w)$ as a process indexed by $w$ which hold uniformly in $w\in \mathcal{W}$ where $\mathcal{W}$ is a set of measures in which the statistician plans to choose $w.$ 
We discuss in Section 3.4\ref{secsecadaptw} the adaptive choice of $w$.\\

\noindent {\bf 2.2.3 The semiparametric case} \label{secsecsecsemi}

Suppose that Model 2 is satisfied.
In the semiparametric case, the family of functions $\mu_{\theta}$ is unknown. However,
the criterion used for the parametric case can be slightly modified to estimate $\theta_0$. We can write
\begin{equation}
\nonumber
\theta_0=\argmin_{\theta \in \Theta} M_w(\theta,\mu_{\theta}),
\end{equation}
where 
$$M_w(\theta,\mu_{\theta})= \int_0^{\tau_H} E\big[\mu_{\theta}(t,\theta 'Z)^2\big]dw(t)-2\int_0^{\tau_H} E\big[Y(t)\mu_{\theta}(t,\theta 'Z)\big]dw(t)$$
and where $w$ is now chosen such that $M_w(\theta,\mu_{\theta})<\infty$ in addition to having finite total mass.

Using a family of nonparametric estimators $\hat{\mu}_{\theta}$ of $\mu_{\theta}$, we define the estimator of $\theta_0$ as
\begin{equation} \label{estimsemipara}
\hat{\theta}(w)=\argmin_{\theta\in \Theta}M_{n,w}(\theta,\hat{\mu}_{\theta}),
\end{equation}
where 
\[M_{n,w}(\theta,\hat{\mu}_{\theta})=n^{-1}\sum_{i=1}^n\int_0^{T_{(n)}} \hat{\mu}_{\theta}(t,\theta 'Z_i)^2dw(t)-2n^{-1}\sum_{i=1}^n\int_0^{T_{(n)}} \hat{Y}_i(t)\hat{\mu}_{\theta}(t,\theta 'Z_i)dw(t).\] 
We give indications in Section 2.2.4\ref{secrajout} on how to perform the minimization 
of such a contrast in practice.
In Section 3.3\ref{secsecsemiparam}, we derive
an asymptotic representation of $\hat{\theta}(w)$ (see Theorem \ref{theoremsemipara}) regardless of the type of nonparametric estimators
$\hat{\mu}_{\theta}$ used in the computation and provided these nonparametric estimators satisfy a list of 
uniform convergence conditions. Nevertheless, let us give a precise example of $\hat{\mu}_{\theta}$ using kernel estimators.
The convergence properties of this type of estimator are derived in Section 6.3\ref{secpreuvenonpara}.

Using the same arguments as in (\ref{NNstar}), we have from the identifiability Assumptions \ref{sym} and \ref{symrec},
\begin{equation}
\mu_{\theta}(t,u)=\int_{0}^t \frac{E[dN(s)|\theta 'Z=u]}{1-G(s-)}.\label{reecmu}
\end{equation}
We estimate the numerator in (\ref{reecmu}) using a kernel estimator and the denominator
by the Kaplan-Meier estimator $\hat{G},$ leading to
\begin{equation}
\label{kernel}
\hat{\mu}_{\theta,h}(t,u)=\int_{0}^t \frac{\sum_{i=1}^nK\left(\frac{\theta 'Z_i-u}{h}\right)dN_i(s)}{\sum_{j=1}^n K\left(\frac{\theta 'Z_j-u}{h}\right)\big(1-\hat{G}(s-)\big)},
\end{equation}
where $K$ is a kernel function and $h$ a bandwidth sequence going to zero.
In Section 6.3\ref{secpreuvenonpara}, we list some conditions on $K$ and $h.$
How to choose the bandwidth from the data in practice is considered at the end of Section 3.6\ref{secsecresnonpara}.

It is important to mention that, in this semiparametric approach, knowledge of the family of functions $\{\mu_{\theta}:\theta\in \Theta\}$ is never required for computing the estimator, since these functions are replaced by nonparametric estimators.
This family of functions will only appear in the theoretical validation of the procedure, as the limit of the estimators $\hat{\mu}_{\theta}.$\\

\noindent {\bf 2.2.4 Minimization of the contrast (\ref{estimsemipara})}\label{secrajout}

The contrast $M_{n,w}(\theta,\hat{\mu}_{\theta})$ can be tricky to minimize in practice, since it
depends on nonparametric estimators. A first possibility consists of using iterative algorithms (see e.g. Xia et al. \citeyearpar{xia2002}).
Another possibility is to use a direct maximization as the one described by H\"ardle et al. \citeyearpar{Hardle1993}
or Delecroix et al. \citeyearpar{Delecroix06} in the case of single-index mean regression. This technique is particularly suited to the 
use of kernel estimators, which depend on a bandwidth parameter $h$ (as the ones described in Section 3.6\ref{secsecresnonpara}
and used in Section 4\ref{secsimu}). 
To emphasize this dependence, we will use the notation $\hat{\mu}_{\theta,h}.$

In this case, the function $M_{n,w}(\theta,\hat{\mu}_{\theta,h})$ can be seen as a function of both $\theta$ and $h.$
H\"ardle et al. \citeyearpar{Hardle1993} proposed to choose jointly $\hat{\theta}$ and an adaptive bandwidth $\hat{h}$
by taking $(\hat{\theta},\hat{h})=\arg \min_{\theta\in \Theta,h\in \mathcal{H}}M_{n,w}(\theta,\hat{\mu}_{\theta,h}),$
where $\mathcal{H}$ denotes a set of bandwidths among which one wishes to select the most appropriate. In Section 4\ref{secsimu},
we use a finite grid of bandwidths $\mathcal{H},$ so that minimizing this contrast with respect to $h$ does not raise any additional technical issue. In the case of mean-regression, H\"ardle et al. \citeyearpar{Hardle1993} have shown that $\hat{h}$ obtained using this technique is asymptotically equivalent to the (uncomputable) $h^*$ obtained using cross-validation if we had an exact knowledge
of the nonparametric part. 
The same result is proved in our context, in Section 3.6.

Nevertheless, the question of initializing the minimization algorithm with the proper starting point
is more delicate. In practice, one may use the average derivative technique, see  Powell et al. \citeyearpar{pss} in the case of mean-regression.
The main advantage of this technique is that it produces closed formulas to compute an estimator of the index $\theta_0.$ 
Therefore, the average derivative technique is often used in order to provide starting points (see e.g. Delecroix et al. \citeyearpar{dhh}). 
Denoting $\nabla_{z}\mu(t|z)$ the vector of partial derivatives of $\mu$ with respect to $z,$ and $\mu'_{\theta_0}(t,u)=\partial_u \mu_{\theta_0}(t,u),$ it follows from the single-index assumption that $\nabla_{z}\mu(t|z)=\theta_0 \mu'_{\theta_0}(t,\theta_0'z).$ Hence, $\int_0^{\tau_H} \nabla_{z}\mu(t|z)dw(t)$ is colinear to $\theta_0.$
This quantity can be estimated by $\int_0^{T_{(n)}} \nabla_{z}\tilde{\mu}(t|z)dw(t),$ where $\tilde{\mu}$ is a nonparametric estimator of $\mu(t|z)$ (which does not take the single-index assumption into account). Computing
$$\beta_n=\frac{1}{n}\sum_{i=1}^n \int_0^{T_{(n)}} \nabla_{z}\tilde{\mu}(t|Z_i)dw(t),$$
this quantity should be close to be colinear to $\theta_0$ provided that $\tilde{\mu}$ is consistent. 
Then, one can compute $\hat{\theta}_{prel}=\beta_n/\beta_{n,1},$  where $\beta_{n,1}$ denotes the first component of $\beta_n.$
This preliminary estimator $\hat{\theta}_{prel},$ has a first component equal to one, as required.

As an estimator $\tilde{\mu},$ one can use
$$
\tilde{\mu}(t|z)=\int_{0}^t \frac{\sum_{i=1}^n\tilde{K}\left(\frac{Z_i-z}{h_0}\right)dN_i(s)}{\sum_{j=1}^n \tilde{K}\left(\frac{ Z_j-z}{h_0}\right)\big(1-\hat{G}(s-)\big)},$$
where $\tilde{K}$ is a (multivariate) kernel function, and $h_0$ a preliminary bandwidth. As every nonparametric estimator in high dimension, the rate of convergence of this estimator may be slow, but it will still be consistent. Moreover, the fact that $\beta_n$
is computed from a mean of nonparametric estimator improves the quality of this preliminary estimation.
To simplify the notations we use the same bandwidth  $h_0$ for each component of $Z_j$ but different bandwidths may be used in practice.
\\

\setcounter{chapter}{3}
\setcounter{equation}{0} %-1
\noindent {\bf 3. Asymptotic results} \label{secasympt}

In this part, we provide asymptotic properties for our estimators.
In Section 3.1\ref{secseclemma}, we expose our main lemma, which is the keystone of our theoretical results. In the next two sections we give asymptotic representations of $\hat{\theta}(w)$ for the parametric and semiparametric models. We then discuss the adaptive choice of the measure $w$ in order to improve the performance of our procedure in Section 3.4\ref{secsecadaptw}. The variance of the limiting process is estimated in Section 3.5\ref{secsecvariance} and the choice of the bandwidth $h$ in (\ref{kernel}) is highlighted in Section 3.6\ref{secsecresnonpara}.

All these results are presented for both models 1 and 2 and a large class of measures $\mathcal W$. The technical assumptions needed for the estimation procedures are listed in Section 6.1\ref{secsecdisc}.
In particular, our results hold true if we consider a class $\{\mu_{\theta}(\cdot,\cdot),\theta \in \Theta\}$ (or $\{\mu_0(\cdot,\cdot;\theta),\theta \in \Theta\}$ in the parametric case) of polynomial functions and if we take $\mathcal W$ as a set of piecewise constant bounded measures with a finite number of jumps.
However, notice that the assumptions presented in Section 6.1\ref{secsecdisc} are more general and correspond to a larger area of practical situations.\\

\noindent {\bf 3.1 The main lemma} \label{secseclemma}
\setcounter{chapter}{3}
\setcounter{section}{0}
\setcounter{equation}{0} 

From a theoretical viewpoint, the main issue stands in studying the difference between $Y$ and its estimated version. The following lemma provides an asymptotic representation for a class of empirical sums in which the process $\hat{Y}$ is involved.

Such kind of asymptotic representations have become very valuable tools for inference in survival analysis, since they allow to transform a non i.i.d. quantity into an other one that can be easily studied using the central limit theorem. See e.g. Stute~\citeyearpar{Stute95}, Van Keilegom and Akritas~\citeyearpar{Van99}, S\'anchez Sellero et al. \citeyearpar{Sanchez05} or Lopez~\citeyearpar{Lopez09} for some similar results in other frameworks.

\begin{lemma} \label{lemmarescale}
Let $\mathcal{F}$ be a class of functions with bounded envelope $\bar F$ satisfying Property \ref{avc}
and assume that Assumptions \ref{aw} and \ref{holder} hold. Define, for any function $f\in \mathcal{F},$
\[S_n(f,w)=\frac{1}{n}\sum_{i=1}^n \int_{0}^{\tau_H}Y_i(t)f(t,Z_i)dw(t)\]
and
\[\hat{S}_n(f,w)=\frac{1}{n}\sum_{i=1}^n \int_{0}^{T_{(n)}}\hat{Y}_i(t)f(t,Z_i)dw(t).\]
\begin{enumerate}[(1)]
\item
Assume that $\sup_{w\in \mathcal{W}}E[S_n(\bar F,w)]<\infty.$
Let 
\[\mathcal I_w(T_i,\delta_i,f)=\int_0^{\tau_H} \int_0^t \eta_{s-}(T_i,\delta_i)E[f(t,Z)d\mu(s|Z)]dw(t)\]

where 
\begin{align*}
d\mu(s|Z) & =\frac{\partial \mu(s|Z)}{\partial s} ds\\
\eta_t(T,\delta) & =\frac{(1-\delta)I(T\leq t)}{1-H(T-)}-\int_0^{t}\frac{I(T\geq s)dG(s)}{[1-H(s-)][1-G(s-)]}
\end{align*}
Then, for all $f\in \mathcal{F},$
\[\hat{S}_n(f,w)-S_n(f,w)=\frac{1}{n}\sum_{i=1}^n \mathcal I_w(T_i,\delta_i,f)
+R_n(f,w),\]
where 
\begin{align*}
\sup_{w\in \mathcal{W},f \in \mathcal{F}}|R_n(f,w)| & =o_P(n^{-1/2}).
\end{align*}
Moreover, if the measures $w$ are all supported in $[0,\tau]$ for some $\tau<\tau_H$, then 
\[\sup_{w\in \mathcal{W},f \in \mathcal{F}}|R_n(f,w)|=O_{P}(n^{-1}\log n).\]
\item 
If $\hat{f}$ denotes a family of nonparametric estimators of functions $f\in \mathcal{F}$ satisfying\\$\sup_{f\in \mathcal{F}}\|\hat{f}-f\|_{\infty}=o_P(1),$ then 
\[\sup_{w\in \mathcal{W}}|\hat{S}_n(\hat{f},w)-\hat{S}_n(f,w)|=o_P(n^{-1/2}).\]
Moreover, if the measures $w$ are all supported in $[0,\tau]$ for some $\tau<\tau_H$, then 
\[\sup_{w\in \mathcal{W}}|\hat{S}_n(\hat{f},w)-\hat{S}_n(f,w)|=O_P(n^{-1}\log n).\]
\end{enumerate}
\end{lemma}

\noindent The proof is postponed to Section 6.2\ref{secgrossproof}. With the estimated rescaled process $\hat{Y}$ at hand, we can now propose
$M-$estimation procedures to estimate the regression function in both the parametric and semiparametric cases.\\

\noindent {\bf 3.2 Asymptotic normality of $\hat{\theta}$ in the parametric case} \label{secsecparam}

Let $\Longrightarrow$ denote the weak convergence.
 
\begin{theorem} \label{theorempara}
Assume that (\ref{modelpara}) holds. Under Assumptions \ref{sym} to \ref{ainversible1}, the estimator in (\ref{estimpara}) admits the following asymptotic representation
\begin{align*}
\hat{\theta}(w)-\theta_0 & = \Sigma_{w,p}^{-1}\left\{\frac{1}{n}\sum_{i=1}^n \left(\int_0^{\tau_H} [Y_i(t)-\mu_0(t,Z_i;\theta_0)]\nabla_{\theta}\mu_0(t,Z_i;\theta_0)dw(t)\right.\right.\\
&\quad+ \left.\left.\int_0^{\tau_H}\!\!\! \int_0^t \eta_{s-}(T_i,\delta_i)E[\nabla_{\theta}\mu_0(t,Z;\theta_0)d\mu_0(s,Z;\theta_0)]dw(t)\right)\right\}+R_n(w),
\end{align*}
where $\sup_{w\in \mathcal{W}}\|R_n(w)\|=o_P(n^{-1/2}).$
As a consequence, for any $w\in \mathcal{W},$
\[\sqrt{n}\big(\hat{\theta}(w)-\theta_0\big)\Longrightarrow \mathcal{N}(0,V_{w,p}),\]
where $V_{w,p}=\Sigma_{w,p}^{-1}\Delta_{w,p} \Sigma_{w,p}^{-1}$ with
\begin{align*}
\Sigma_{w,p} & = \nabla^2_{\theta}M_w(\theta_0,\mu_0),\\
\Delta_{w,p} & = E\left[\mathcal I_w(T,\delta,\nabla_{\theta}\mu_0(\cdot,\cdot;\theta_0))\mathcal I_w(T,\delta,\nabla_{\theta}\mu_0(\cdot,\cdot;\theta_0))'\right].
\end{align*} 
\end{theorem}

\begin{proof}
Write
\begin{equation}\label{ralebol}
M_{n,w}(\theta,\mu_0)=-2\hat{S}_n(\mu_0(\cdot,\cdot;\theta),w)+n^{-1}\sum_{i=1}^n \int_0^{T_{(n)}} \mu_{0}(t,Z_i;\theta)^2dw(t).
\end{equation}
Then, using the asymptotic representation of Lemma \ref{lemmarescale}, one gets
\begin{align*}
M_{n,w}(\theta,\mu_0) &=  -\frac{2}{n}\sum_{i=1}^n \int_{0}^{T_{(n)}}Y_i(t)\mu_0(t,Z_i;\theta)dw(t)\\ 
&  +\frac{1}{n}\sum_{i=1}^n \int_0^{T_{(n)}} \mu_{0}(t,Z_i;\theta)^2dw(t) \\
 & -\frac{2}{n}\sum_{i=1}^n \int_0^{\tau_H} \int_0^t \eta_{s-}(T_i,\delta_i)E[\mu_0(t,Z;\theta)d\mu(s|Z)]dw(t)+R_n(\mu_0(\cdot,\cdot;\theta),w) \\
 &= \mathcal{T}_1(\theta,w)+\mathcal{T}_2(\theta,w)+\mathcal{T}_3(\theta,w)+R_n(\mu_0(\cdot,\cdot;\theta),w),
\end{align*}
where $\sup_{\theta,w} |R_n(\mu_0(\cdot,\cdot,\theta),w)|=o_P(1),$ since Assumption \ref{ainversible1} ensures that \\$\{\mu_0(\cdot,\cdot;\theta),\theta \in \Theta\}$ satisfies the conditions of Lemma \ref{lemmarescale}. The sum of the first two terms $\mathcal{T}_1(\theta,w)+\mathcal{T}_2(\theta,w)$ converges towards
$M_{w}(\theta,\mu_0)$ from the law of large numbers. Moreover, Assumption \ref{ainversible1} ensures that this convergence is uniform with respect to $\theta$ and $w$. On the other hand, the expectation of $\mathcal{T}_3(\theta,w)$ is 0, and the convergence of $\mathcal{T}_3(\theta,w)$ is uniform with respect to $\theta$ and $w$ since it is of the form $n^{-1}\sum_{i=1}^n \int_0^{\tau_H} \psi(\theta,t,T_i,\delta_i,Z_i)dw(t),$
with $\psi$ satisfying a Lipschitz property with respect to $\theta$ (this is a consequence of Properties \ref{aclasse} and \ref{lipsch} for $\{\mu_0(\cdot,\cdot;\theta),\theta \in \Theta\}$). Finally, this ensures $\sup_{\theta,w}|M_{n,w}(\theta,\mu_0)-M_w(\theta,\mu_0)|=o_P(1)$ and consequently $\sup_w\|\hat{\theta}(w)-\theta_0\|=o_P(1).$

To obtain the uniform CLT property for $\hat{\theta}(w)$, use a Taylor expansion of $\nabla_{\theta}M_{n,w}(\theta,\mu_0)$ around $\theta_0$:
\begin{equation}
\label{taylor1}
\nabla_{\theta}M_{n,w}(\hat{\theta},\mu_0)=\nabla_{\theta}M_{n,w}(\theta_0,\mu_0)+\nabla^2_{\theta}M_{n,w}(\tilde{\theta},\mu_0)(\hat{\theta}-\theta_0),
\end{equation}
for some $\tilde{\theta}$ between $\hat{\theta}$ and $\theta_0.$ The left-hand side of (\ref{taylor1}) is zero by definition
of $\hat{\theta}.$ Moreover, the matrix $\nabla^2_{\theta}M_{n,w}(\tilde{\theta},\mu_0)$ is almost surely invertible for $n$
large enough under Assumption \ref{ainversible1} since $\hat{\theta}$ (and consequently $\tilde{\theta}$) tends to $\theta_0$ almost surely. This leads to
\[\hat{\theta}-\theta_0=-\nabla^2_{\theta}M_{n,w}^{-1}(\tilde{\theta},\mu_0)\nabla_{\theta}M_{n,w}(\theta_0,\mu_0).\]
Write
\begin{align*}
\nabla^2_{\theta}M_{n,w}(\tilde{\theta},\mu_0)=
& -2\bigg[\hat{S}_n(\nabla^2_{\theta}\mu_0(\cdot,\cdot;\tilde{\theta}),w)-\frac{1}{n}\sum_{i=1}^n \int_0^{\tau_H}\left( \nabla_{\theta}\mu_0(t,Z_i;\tilde{\theta})\nabla_{\theta}\mu_0(t,Z_i;\tilde{\theta})' \right. \\
& \qquad\qquad \left. +\mu_0(t,Z_i;\tilde{\theta})\nabla_{\theta}^2 \mu_0(t,Z_i;\tilde{\theta})\right)dw(t)\bigg]+R_n(\theta,w),
\end{align*}
where $R_n(\theta,w)$ comes from the change in the integration bounds of $[0,T_{(n)}]$ by $[0,\tau_H]$ and can be seen to tend uniformly to zero from Lebesgue's dominated convergence since the term inside the integral is bounded.
From Lemma \ref{lemmarescale}, the almost sure convergence of $\tilde{\theta}$ and the fact that 
$\{\nabla^2_{\theta}\mu_0(\cdot,\cdot,\theta), \theta \in \Theta\}$ satisfies Property \ref{lipsch} (see Assumption \ref{ainversible1}), we get that $\hat{S}_n(\nabla^2_{\theta}\mu_0(\cdot,\cdot;\tilde{\theta}),w)$ converges to $\int_0^{\tau_H} E[Y(t)\nabla_{\theta}^2\mu_0(t,Z;\theta_0)]dw(t)$ uniformly in $w$. The second part converges uniformly to its expectation over $\Theta$
as a consequence of the Glivenko-Cantelli property of classes of functions satisfying Property \ref{lipsch}.
This shows that 
\[\sup_{w}\|\nabla^2_{\theta}M_{n,w}^{-1}(\tilde{\theta},\mu_0)-\nabla_{\theta}^2 M^{-1}_w(\theta_0,\mu_0)\|=o_P(1).\]
On the other hand, we write
\begin{align*}
\nabla_{\theta}M_{n,w}(\theta_0,\mu_0) & =-2\left[\hat{S}_n(\nabla_{\theta}\mu_0(\cdot,\cdot;\theta_0),w)-\frac{1}{n}\sum_{i=1}^n\int_0^{\tau_H} \mu_0(t,Z;\theta_0)\nabla_{\theta}\mu_0(t,Z;\theta_0)dw(t)\right]\\
& \quad +\frac{2}{n}\sum_{i=1}^n\int_{T_{(n)}}^{\tau_H} \mu_0(t,Z;\theta_0)\nabla_{\theta}\mu_0(t,Z;\theta_0)dw(t).
\end{align*}
Using Lebesgue's dominated convergence theorem, the last term tends uniformly to zero at a $n^{-1/2}$ rate. Finally, the asymptotic representation follows from Lemma \ref{lemmarescale}.
\end{proof}

\noindent {\bf 3.3 Asymptotic normality of $\hat{\theta}$ in the semiparametric case} \label{secsecsemiparam}

\begin{theorem} \label{theoremsemipara}
Assume that (\ref{modelsemipara}) holds. Under Assumptions \ref{sym} to \ref{holder} and \ref{ainversible2} to \ref{aregularite}, the estimator in (\ref{estimsemipara}) admits the following asymptotic representation
\begin{align*}
\hat{\theta}(w)-\theta_0 & = \Sigma_{w,sp}^{-1}\left\{\frac{1}{n}\sum_{i=1}^n \left(\int_0^{\tau_H} [Y_i(t)-\mu_{\theta_0}(t,\theta_0'Z_i)]\nabla_{\theta}\mu_{\theta_0}(t,Z_i)dw(t)\right.\right.\\
&\quad+ \left.\left.\int_0^{\tau_H}\!\!\! \int_0^t \eta_{s-}(T_i,\delta_i)E[\nabla_{\theta}\mu_{\theta_0}(t,Z)d\mu_{\theta_0}(s,\theta_0'Z)]dw(t)\right)\right\}+R_n(w),
\end{align*}
where $\sup_{w\in \mathcal{W}}\|R_n(w)\|=o_P(n^{-1/2}).$
As a consequence, for any $w\in \mathcal{W},$
\[\sqrt{n}(\hat{\theta}(w)-\theta_0)\Longrightarrow \mathcal{N}(0,V_{w,sp}),\]
where $V_{w,sp}=\Sigma_{w,sp}^{-1}\Delta_{w,sp} \Sigma_{w,sp}^{-1}$ with
\begin{align*}
\Sigma_{w,p} & = \nabla^2_{\theta}M_w(\theta_0,\mu_{\theta_0}),\\
\Delta_{w,p} & = E\left[\mathcal I_w(T,\delta,\nabla_{\theta}\mu_{\theta_0}(\cdot,\cdot;\theta_0))\mathcal I_w(T,\delta,\nabla_{\theta}\mu_{\theta_0}(\cdot,\cdot;\theta_0))'\right].
\end{align*}

\end{theorem}

\begin{proof}
 The consistency of the preliminary estimator can be proved in the same way as in the proof of Theorem \ref{theorempara}. Following the decomposition (\ref{ralebol}), one can show that
$M_{n,w}(\theta, \hat{\mu}_{\theta})=M_{n,w}(\mu_{\theta})+R_n(\theta,w),$
where $\sup_{\theta,w}|R_n(\theta,w)|=o_P(1).$
Indeed, the second part of Lemma \ref{lemmarescale} allows us to replace $\hat{S}_n(\hat{\mu}_{\theta},w)$ by $S_n(\hat{\mu}_{\theta},w)$ up to some uniformly negligible remainder term.
Next, observe that
\begin{align*}\frac{1}{n}\sum_{i=1}^n \int_0^{T_{(n)}} \hat{\mu}_{\theta}(t,\theta'Z_i)^2dw(t)& =\frac{1}{n}\sum_{i=1}^n \int_0^{T_{(n)}} \mu_{\theta}(t,\theta'Z_i)^2dw(t)\\
&\quad +\frac{1}{n}\sum_{i=1}^n \int_0^{T_{(n)}}  [\hat{\mu}_{\theta}(t,\theta'Z_i)^2-\mu_{\theta}(t,\theta'Z_i)^2]dw(t),
\end{align*}
where the second term goes to zero uniformly in $\theta$ and $w$ thanks to the uniform convergence assumptions on $\hat{\mu}_{\theta}.$ This shows that $\sup_{\theta,w}|M_{n,w}(\theta, \hat{\mu}_{\theta})-M_{n,w}(\theta,\mu_{\theta})|=o_P(1)$ and the uniform convergence of $M_{n,w}(\theta,\mu_{\theta})$ is obtained following the path of the proof of Theorem \ref{theorempara}.

Asymptotic normality comes from the fact that
\[\hat{\theta}-\theta_0=-\nabla^2_{\theta}M_{n,w}^{-1}(\tilde{\theta},\hat{\mu}_{\tilde{\theta}})\nabla_{\theta}M_{n,w}(\theta_0,\hat{\mu}_{\theta_0}).\] 
The fact that
\[\sup_{w}\|\nabla^2_{\theta}M_{n,w}^{-1}(\tilde{\theta},\hat{\mu}_{\tilde{\theta}})-\nabla_{\theta}^2 M^{-1}_w(\theta_0,\mu_{\theta_0})\|=o_P(1)\]
can be shown in the same way as in the proof of Theorem \ref{theorempara} using now the second part of Lemma \ref{lemmarescale}. The big issue consists of proving the asymptotic representation of $\nabla_{\theta}M_{n,w}(\theta_0,\hat{\mu}_{\theta_0})$. Write
\[\nabla_{\theta}M_{n,w}(\theta_0,\hat{\mu}_{\theta_0})=
-2\left[\hat{S}_n(\nabla_{\theta}\hat{\mu}_{\theta_0}(\cdot,\theta_0'\cdot),w)-\frac{1}{n}\sum_{i=1}^n \int_0^{T_{(n)}} \hat{\mu}_{\theta_0}(t,\theta_0'Z_i)\nabla_{\theta}\hat{\mu}_{\theta_0}(t,Z_i) dw(t)\right].\]
Using the second part of Lemma \ref{lemmarescale}, this can be rewritten as
\begin{align*}
& \nabla_{\theta}M_{n,w}(\theta_0,\hat{\mu}_{\theta_0})\\
&\quad = \nabla_{\theta}M_{n,w}(\theta_0,\mu_{\theta_0})\\
&\qquad-\frac{2}{n}\sum_{i=1}^n \int_0^{\tau_H}\bar{\mu}_{\theta_0}(t,\theta_0'Z_i)^{\lambda_1+\lambda_2} \big(\mu_{\theta_0}(t,\theta_0'Z_i)-Y_i(t)\big)\frac{\nabla_{\theta}\mu_{\theta_0}(t,Z_i)-\nabla_{\theta}\hat{\mu}_{\theta_0}(t,Z_i)}{\bar{\mu}_{\theta_0}(t,\theta_0'Z_i)^{\lambda_1+\lambda_2}}dw(t) \\
& \qquad+\frac{2}{n}\sum_{i=1}^n \int_0^{\tau_H}\frac{\hat{\mu}_{\theta_0}(t,\theta_0'Z_i)-\mu_{\theta_0}(t,\theta_0'Z_i)}{\bar{\mu}_{\theta_0}(t,\theta_0'Z_i)^{\lambda_1+\lambda_2}}\bar{\mu}_{\theta_0}(t,\theta_0'Z_i)^{\lambda_1+\lambda_2}\nabla_{\theta}\mu_{\theta_0}(t,Z_i)dw(t) \\
& \qquad+\frac{2}{n}\sum_{i=1}^n \int_0^{\tau_H} \frac{\big(\hat{\mu}_{\theta_0}(t,\theta_0'Z_i)-\mu_{\theta_0}(t,\theta_0'Z_i)\big)\big(\nabla_{\theta}\hat{\mu}_{\theta_0}(t,Z_i)-\nabla_{\theta}\mu_{\theta_0}(t,Z_i)\big)}
{\bar{\mu}_{\theta_0}(t,\theta_0'Z_i)^{2(\lambda_1+\lambda_2)}\bar{\mu}_{\theta_0}(t,\theta_0'Z_i)^{-2(\lambda_1+\lambda_2)}}dw(t) \\
&\qquad +R_{4n}(w)\\
&\quad= \nabla_{\theta}M_{n,w}(\theta_0,\mu_{\theta_0}) + R_{1n}(w)+R_{2n}(w)+R_{3n}(w)+R_{4n}(w),
\end{align*}
where $R_{4n}(w)$ comes from Lemma \ref{lemmarescale} and the change in the integration bound of $[0,T_{(n)}]$ by $[0,\tau_H].$ Using the same arguments as in the proof of Theorem \ref{theorempara}, we deduce that 
$\sup_{w}\|R_{4n}(w)\|=o_P(n^{-1/2}).$
Using the uniform convergence rates of $\hat{\mu}_{\theta_0}$ and of $\nabla_{\theta}\hat{\mu}_{\theta_0}$, we get straightforwardly that 
$\sup_{w}\|R_{3n}(w)\|=o_P(n^{-1/2}).$
Using the uniform convergence of $\nabla_{\theta}\hat{\mu}_{\theta_0}$, we see that the term $R_{1n}$ can be decomposed into
\[R_{1n}(w)=n^{-1}\sum_{i=1}^n \big(f_w(Z_i,Y_i)-f_{n,w}(Z_i,Y_i)\big),\]
where $f_w$ and $f_{n,w}$ both belong (almost surely for $n$ large enough) to the class $\mathcal{G}$ defined in Assumption \ref{aregularite} and with $\sup_{w}\|f_w-f_{n,w}\|_{\infty}\rightarrow 0$ a.s. Therefore, using the asymptotic equicontinuity of the Donsker class $\mathcal{G}$ (see e.g. Section 2.1.2 in Van der Vaart and Wellner~\citeyearpar{Vaart96}), this shows that
\[\sup_w\|R_{1n}(w)-\int \big(f_w(z,y)-f_{n,w}(z,y)\big)dP_{Z,Y}(z,y)\|=o_P(n^{-1/2}).\]
Moreover, it is clear that $\int (f_w(z,y)-f_{n,w}(z,y))dP_{Z,Y}(z,y)=0$ using the fact that $\nabla_{\theta}\mu_{\theta_0}(t,z)-\nabla_{\theta}\hat{\mu}_{\theta_0}(t,z)$ is a function of $z$ only and that
$E[\mu_{\theta_0}(t,\theta_0'Z_i)-Y_i(t)|Z_i]=0$.

The term $R_{2n}(w)$ can be handled in the same way using now the Donsker class $\mathcal{H}$ in Assumption \ref{aregularite}, observing that 
$\hat{\mu}_{\theta_0}(t,\theta_0'z)-\mu_{\theta_0}(t,\theta_0'z)$ is a function of $\theta_0'z$ only and getting from Lemma 5 in Supplementary Material, that $E[\nabla_{\theta}\mu_{\theta_0}(t,Z)|\theta_0'Z]=0$.%\ref{lemmacalculgradient}
\end{proof}

\noindent {\bf 3.4 Adaptive choice of $w$} \label{secsecadaptw}

The representations of Theorems \ref{theorempara} and \ref{theoremsemipara} hold uniformly in $w\in \mathcal{W}.$
Therefore, consider some data-driven measure $\hat{w}_n\in \mathcal{W},$ tending to some asymptotic measure $w_0$ in the sense that
$\int \phi(t)d\{\hat{w}_n-w_0\}(t)$ tends to $0$ in probability for all function $\phi$ in $L^1(d\hat{w}_n)\cap L^1(dw_0).$
Moreover, assume that
\begin{equation}
\label{ralebol2}\sup_{w\in \mathcal{W}:w\rightarrow w_0}\left|n^{-1}\sum_{i=1}^n \int \phi(t;T_i,\delta_i,Z_i,Y_i)d\{w-w_0\}(t)\right|=o_P(n^{-1/2}),
\end{equation}
for any function $\phi$ with $E[\phi(t;T_i,\delta_i,Z_i,Y_i)]=0.$
This assumption can easily be fulfilled by considering a simple class of measures, ensuring that $\{(T,Z,\delta,Y)\rightarrow n^{-1}\sum_{i=1}^n \int \phi(t;T,\delta,Z,Y)dw(t):w\in \mathcal{W}\}$ is a Donsker class of functions.
In this situation, one can deduce that (for example in the semiparametric case),
\begin{align}\label{gouttedeau}
\hat{\theta}(\hat{w})-\theta_0 & = \Sigma_{w,sp}^{-1}\left\{\frac{1}{n}\sum_{i=1}^n \left(\int_0^{\tau_H} [Y_i(t)-\mu_{\theta_0}(t,\theta_0'Z_i)]\nabla_{\theta}\mu_{\theta_0}(t,Z_i)dw_0(t)\right.\right. \nonumber\\
&\quad+ \left.\left.\int_0^{\tau_H}\!\!\! \int_0^t \eta_{s-}(T_i,\delta_i)E[\nabla_{\theta}\mu_{\theta_0}(t,Z)d\mu_{\theta_0}(s,\theta_0'Z)]dw_0(t)\right)\right\}+R'_n(\hat{w}), 
\end{align}
where $R'_n(\hat{w})=o_P(n^{-1/2}).$ 
This can be done by combining the asymptotic representation of $\hat{\theta}(\hat{w})-\hat\theta(w_0)$ and of $\hat{\theta}(w_0)-\theta_0,$ and using (\ref{ralebol2}) by studying the difference between the two main terms.

The idea is to consider an asymptotically optimal measure $w_0.$ For example, suppose that one is searching for an estimator that minimizes
$E[\|\hat{\theta}-\theta_0\|^2].$ Denote $V_{w,sp}(i)$ the $i-$th diagonal element of the asymptotic covariance matrix $V_{w,sp}.$
In view of our objective, one wishes to consider $w_0$ as the minimizer among measures $\mathcal{W}$ of $n^{-1}\sum_{i=1}^d V_{w,sp}(i).$ 
In practice, $\hat{w}$ can be taken as the minimizer over $\mathcal{W}$ of $n^{-1}\sum_{i=1}^d \hat{V}_{w,sp}(i),$ where $\hat{V}_{w,sp}$ denotes an estimated version of the asymptotic covariance matrix. 
The quantity $n^{-1}\sum_{i=1}^d \hat{V}_{w,sp}(i)$ can be seen as an estimator of the mean-squared error of $\hat{\theta}(w).$ 
Such an adaptive weight function will converge towards $w_0$ under mild conditions, provided that there is some continuity over $\mathcal{W}$ of the map $w\rightarrow V_{w,sp}.$ 
Then, representation (\ref{gouttedeau}) ensures that $\hat{\theta}(\hat{w})$ is asymptotically equivalent to the optimal estimator $\hat{\theta}(w_0).$\\

\noindent {\bf 3.5 Estimation of the variance} \label{secsecvariance}

We show how to estimate the variance in the representation of Theorem \ref{theoremsemipara} and we propose an estimator of the mean squared error of $\theta_0$. Denote by $\xi_{n,w}$ the term between brackets in the representation of Theorem \ref{theoremsemipara} so that
\[\hat{\theta}(w)-\theta_0 = \Sigma_{w,sp}^{-1}\,\xi_{n,w}+R_n(w),\]
where $\sup_{w\in \mathcal{W}}\|R_n(w)\|=o_P(n^{-1/2})$.
The quantity $\xi_{n,w}$ can be estimated in the following way 
\[\hat{\xi}_{n,w}=\frac{1}{n}\sum_{i=1}^n \hat{\psi}(\delta_i,Z_i,T_i,Y_i;w),\]
where
\begin{align*}
\hat{\psi}(\delta,Z,T,Y;w) & = \int_0^{T_{(n)}} \big(Y(t)-\hat{\mu}_{\hat{\theta}}(t,\hat{\theta}'Z)\big)\nabla_{\theta}
\hat{\mu}_{\hat{\theta}}(t,Z)dw(t)\\
&\quad+ \int_0^{T_{(n)}}\!\!\! \int_0^t \hat{\eta}_{s-}(T,\delta)n^{-1}\sum_{i=1}^n\,\big(\nabla_{\theta}\hat{\mu}_{\hat{\theta}}(t,Z_i)d\hat{\mu}_{\hat{\theta}}
(s,\hat{\theta}'Z_i)\big)dw(t),
\end{align*}
\[\hat{\eta}_t(T,\delta) =\frac{(1-\delta)I(T\leq t)}{1-\hat{H}(T-)}-\int_0^{t}\frac{I(T\geq s)d\hat{G}(s)}{\big(1-\hat{H}(s-)\big)\big(1-\hat{G}(s-)\big)}\]
and $\hat{H}$ is the empirical estimator of $H$.

Therefore, the quantity $\Delta_{w,sp}$ can be estimated by
\[\hat{\Delta}_{w,sp}=\frac{1}{n}\sum_{i=1}^n\left(\hat{\psi}(\delta,Z,T,Y;w) -\frac{1}{n}\sum_{i=1}^n \hat{\psi}(\delta,Z,T,Y;w)\right)^{\otimes 2},\]
where $\otimes 2$ denotes the product of the matrix with its transpose. 
To consistently estimate $\Sigma_{w,sp}$, we use
\[\hat{\Sigma}_{w,sp}=\frac{1}{n}\sum_{i=1}^n\int_0^{T_{(n)}}\nabla_{\theta} \hat{\mu}_{\hat{\theta}}(t,Z_i)\nabla_{\theta}\hat{\mu}_{\hat{\theta}}(t,Z_i)' dw(t).\]
A consistent estimator of $V_{w,sp}$ can then be computed from $\hat{V}_{w,sp}=\hat{\Sigma}_{w,sp}^{-1}\,\hat{\Delta}_{w,sp}\,\hat{\Sigma}_{w,sp}^{-1}$. 
The consistency of $\hat{V}_{w,sp}$ comes from the uniform consistency of the estimator $\hat{H}$ and $\hat{G}$ (and of Kaplan-Meier integrals with respect to $\hat{G}$), from the consistency of $\hat{\theta},$ and from the uniform consistency of the nonparametric estimators of $\hat{\mu}_{\theta}$ and of its partial derivatives.\\

\noindent {\bf 3.6 Estimation of the nonparametric part} \label{secsecresnonpara}

In the semiparametric model, estimation of the finite dimensional parameter
$\theta_0$ is only the first step of the method. With our estimator $\hat{\theta}$ at hand, we wish to
estimate the conditional mean function $\mu(t|z).$ Different strategies can be proposed to perform
this estimation. For this final estimator, there is no theoretical need to use the same kind of nonparametric estimator as
in the computation of $\hat{\theta}.$ Proposition \ref{propnonpara} below states that, under some convergence
assumptions for the nonparametric estimator used in this second step, the asymptotic behavior of the final semiparametric estimator of $\mu$ is identical to the asymptotic behavior of a purely nonparametric estimator in the case where $\theta_0$
is exactly known.

\begin{proposition} \label{propnonpara}
Let $\Theta^*$ be some neighborhood of $\theta_0,$ and let $\mathcal{T}$ be a set on which \linebreak $\sup_{\theta\in \Theta^*,t\in \mathcal{T},z\in \mathcal{Z}}\|\nabla_{\theta}\mu_{\theta_0}(t,z)\|<\infty.$
Let $\hat{\mu}_{\theta}$ be a family of nonparametric estimators of $\mu_{\theta}$ satisfying the assumption
\begin{equation}\label{easy}\sup_{\theta\in \Theta^*,t\in \mathcal{T},z\in \mathcal{Z}}\|\nabla_{\theta}\hat{\mu}_{\theta}(t,z)-\nabla_{\theta}\mu_{\theta}(t,z)\|=o_P(1).
\end{equation}
Then, we have
\[\sup_{t \in \mathcal{T},z\in \mathcal{Z}}|\hat{\mu}_{\hat{\theta}}(t,\hat{\theta}'z)-\hat{\mu}_{\theta_0}(t,\theta_0'z)|=O_P(n^{-1/2}).\]
\end{proposition}

\begin{proof}
This is a direct consequence of a Taylor expansion of $\hat{\mu}_{\hat{\theta}}$ around $\theta_0.$ From Theorem \ref{theoremsemipara} we have $\hat{\theta}-\theta_0=O_P(n^{-1/2}).$ Then, the boundedness of $\nabla_{\theta}\mu_{\theta_0}(t,z)$
and the uniform convergence in assumption (\ref{easy}) give the result.
\end{proof}

As explained in Section 2.2.4, we propose to simultaneously select the adaptive bandwidth $\hat{h}$ and the index parameter $\hat \theta$ in the following way:
\begin{align}\label{double}
(\hat{\theta},\hat{h})=\argmin_{\theta \in \Theta, h\in \mathcal{H}} M_{n,w}(\theta,\hat{\mu}_{\theta,h}).
\end{align}
The uniform in bandwidth consistency of the kernel estimators we use (see Section 6.3\ref{secpreuvenonpara}) ensures us that $\hat{\theta}$ has the same asymptotic properties as in Theorem \ref{theoremsemipara}.
On the other hand, Proposition \ref{proph} below shows that the adaptive bandwidth $\hat{h}$ is asymptotically equivalent to the bandwidth we could obtain using a classical cross-validation technique in the case where the parameter $\theta_0$ is exactly known.

\begin{proposition} \label{proph}
For some positive constants $a,b$ and $\alpha$, let $\mathcal{H}=[an^{-\alpha},bn^{-\alpha}]$ be a set of bandwidths satisfying Assumption \ref{kern} and let
\[h_0=\argmin_{h\in \mathcal{H}} M_{n,w}(\theta_0,\hat{\mu}_{\theta_0,h}).\]
Under the assumptions of Theorem \ref{theoremsemipara} and provided that $\sup_{h\in \mathcal{H},t\in \mathbb{R}_+, z\in \mathcal{Z}}|\hat{\mu}_{\theta,h}(t,\theta 'z)-\mu_{\theta,h}(t,\theta 'z)|=o_{P}(1)$, we have
\[\hat{h}/h_0\rightarrow 1 \text{ a.s}.\]
\end{proposition} 

\begin{proof}
Define $\phi(h/h_0)=M_{n,w}(\theta_0,\hat{\mu}_{\theta_0,h})$ and $\phi_n(h/h_0)=\argmin_{\theta\in \Theta}M_{n,w}(\theta,\hat{\mu}_{\theta,h})$.
By definition of $h_0$ and $\hat{h}$ we have $\argmin_{x\in [a,b]}\phi(x)=1$ and $\hat{h}/h_0=\argmin_{x\in [a,b]}\phi_n(x)$.
Now write
\begin{align*}
\phi_n(x) &=\phi(x)-\frac{2}{n}\sum_{i=1}^n \int_0^{\tau_H} \hat{Y}_i(t)\big(\hat{\mu}_{\theta,xh_0}(t,\theta 'Z_i)-\hat{\mu}_{\theta,h_0}(t,\theta 'Z_i)\big)dw(t) \\
&\quad + \frac{1}{n}\sum_{i=1}^n \int_0^{\tau_H} \big(\hat{\mu}_{\theta,xh_0}(t,\theta 'Z_i)^2-\hat{\mu}_{\theta,h_0}(t,\theta 'Z_i)^2\big)dw(t)+M_{n,w}(\theta,\hat{\mu}_{\theta,h_0}).
\end{align*}
Using Lemma \ref{lemmarescale} and the uniform in bandwidth consistency of $\hat{\mu}_{\theta,h}$,
the second and third terms in the decomposition tend to zero uniformly in $x$. On the other hand, the last term
does not depend on $x$. This shows that $\hat{h}/h_0\rightarrow 1$ a.s.
\end{proof}

\noindent {\bf 4 Simulations} \label{secsimu}
\setcounter{chapter}{4}
\setcounter{equation}{0} 

We present here some empirical evidence of the good behavior of our semiparametric estimation procedure for finite sample sizes.

\noindent The variables $D_i$ are generated according to a Weibull distribution with parameter $(10,1.1)$ and we consider 4-dimensional covariates $Z_i\sim \otimes^4 \mathcal{U}[1,2]$ for $i=1,\ldots,n$. Conditionally on $Z_i$, the processes $N_i^*(\cdot)$ are generated in the following way: we first generate homogeneous Poisson processes $\widetilde N_i(\cdot)$ with intensity $\theta_0'Z_i+5$ where $\theta_0=(1,1.6,1.25,0.7)'$ and then compute $N_i^*(t)=\widetilde N_i(t\wedge D_i)$ for $i=1,\ldots,n$. This ensures that
\begin{equation*}E[N_i^*(t)|Z_i]=(\theta_0'Z_i+5)\int_0^t (1-F(s-))ds , \quad i=1, \ldots,n,\end{equation*}
%{\color{red} In this framework, a quick computation shows that $\mu_{\theta}(t,u)=(\theta'_0E[Z_i|\theta'Z_i=u]+5)t.$
% Indeed,
% \begin{eqnarray*}
% E\left[N_i^*(t)|\theta'Z_i\right] &=& E\left[E[N_i^*(t)|Z_i]|\theta'Z_i\right] \\
% &=& (\theta_0'E[Z_i|\theta'Z_i]+5)t,
% \end{eqnarray*}
% where we used (\ref{idiot}) for line 2.
% }
such that condition (\ref{modelsemipara}) is verified.
The censoring distribution is selected to be Weibull with parameters $(4,\lambda)$. Taking $\lambda=1.38$ or $\lambda=1$ leads respectively to $30\%$ or $70\%$ of censoring and an average of $11.5$ or $10$ recurrent events per sample. We decide to estimate $\mu_{\theta}$ with the kernel estimator defined in (\ref{kernel}) with the Epanechnikov kernel.
In our simulation study, we emphasize the impact of the two parameters involved in our semiparametric procedure,
namely the bandwidth of the nonparametric kernel estimators and the measure $w.$

%\noindent In our simulation study, we consider the case where, conditionally on $Z_i,$ the process $N_i^*$ is an homogeneous Poisson process with intensity $\theta_0'Z_i+5,$ that is 
%\begin{equation}\label{idiot}E[N_i^*(t)|Z_i]=(\theta_0'Z_i+5)t, \quad i=1, \ldots,n.\end{equation}
%%{\color{red} In this framework, a quick computation shows that $\mu_{\theta}(t,u)=(\theta'_0E[Z_i|\theta'Z_i=u]+5)t.$
%% Indeed,
%% \begin{eqnarray*}
%% E\left[N_i^*(t)|\theta'Z_i\right] &=& E\left[E[N_i^*(t)|Z_i]|\theta'Z_i\right] \\
%% &=& (\theta_0'E[Z_i|\theta'Z_i]+5)t,
%% \end{eqnarray*}
%% where we used (\ref{idiot}) for line 2.
%% }

%We take $\theta_0=(1,1.6,1.25,0.7)'$ and we consider 4-dimensional covariates $Z_i\sim \otimes^4 \mathcal{U}[1,2]$ for %$i=1,...,n$. The variables $D_i$ for $i=1,...,n$ are generated according to a Weibull distribution with parameters $(10,1.09)$.
%The censoring distribution is selected to be Weibull with parameters $(4,\lambda)$. Taking $\lambda=1.38$ or $\lambda=1$ leads respectively to $30\%$ or $70\%$ of censoring and an average of $11.5$ or $10$ recurrent events per sample. We decide to estimate $\mu_{\theta}$ with the kernel estimator defined in (\ref{kernel}) with the Epanechnikov kernel.
%In our simulation study, we emphasize the impact of the two parameters involved in our semiparametric procedure,
%namely the bandwidth of the nonparametric kernel estimators and the measure $w.$

First, we consider the case of a fixed bandwidth and show how the adaptive choice of $\hat{w}$ can improve the estimation performance of the parameter $\theta_0$. The kernel estimator is computed using a bandwidth $h_0=1.1$. We consider a set of discrete measures $w(\cdot)$ supported on $\mathcal{I}=\{0.1,0.2,\ldots,1.2\}.$ The range of values of $\mathcal I$ is chosen accordingly to the range of values of the recurrent event times such that $\mathcal{I}$ is a representative subset of the whole trajectory of $N(\cdot)$. With this choice, for any function $f$, the integral with respect to $w$ reduces to a finite sum. Indeed, we have 
\[\int f(t)dw(t)=\sum_{k\in \mathcal{I}}f(k)w(\{k\}).\]
Moreover, we consider an adaptive choice of the weight $w$ among the following family $\{w_a, a=(a_1,\ldots,a_4) \text{ with } a_i\in \{0.25,0.5,0.75,1\},i=1,\ldots,4\}$ with

\[ w_a(\{k\})=\begin{cases}
1 \quad \text{ for } k=0.1,\ldots,0.8\\
  a_i \quad \text{ for } k=0.8+i, i=1,\ldots,4
\end{cases}\]
which makes $256$ possible choices. The intuition is that our procedure should allocate smaller weights to large values of $T_i$ since the behaviour of the Kaplan-Meier estimator is known to be less effective in this part of the distribution (and contributes significantly to the variance). Our estimator $\hat{\theta}_{\hat{w},h_0}$ is then compared to the estimator $\tilde{\theta}$ obtained for the measure $w_0$ which puts mass $1$ at every point of $\mathcal I$. We also compare our estimator $\hat{\theta}_{\hat{w},h_0}$ to the estimators $\hat{\theta}_{\mathrm cox}$ and $\hat{\theta}_{\mathrm AFT}$ obtained respectively under Model \eqref{E:Cox} and Model \eqref{E:AFT2} by Andersen and Gill \citeyearpar{AndersenGill1982} and Ghosh \citeyearpar{Ghosh2004}.

In Tables 1 and 2, we report our results over $1\,000$ simulations of samples of size $100$ for two different rates of censoring ($p=30\%$ and $p=70\%$). Recalling that the first component of $\theta_0$ is imposed to be one, we only have to estimate the three other components. For each estimator, the Mean Squared Error (MSE) $ E(\|\hat{\theta}-\theta_0\|^2)$ is decomposed into bias and variance.

\begin{table}
\caption{Biases, variances and MSE of $\tilde{\theta},\hat {\theta}_{\hat w,h_0},\hat{\theta}_{\mathrm cox}$ and $\hat{\theta}_{\mathrm AFT}$ for $30\%$ of censored data}
\label{30w}
%\begin{center}
\centering
\begin{tabular}{|l|c|c|c|}
\hline $p=30\%$ & Bias & Variance & MSE \\
\hline \hline $\tilde{\theta}$ & $\left(\begin{array}{c}
0.0816\\
0.0897\\
0.0289 \end{array}\right)$
 & $\left(\begin{array}{ccc}
0.0938 & -0.0122 & -0.0369 \\
-0.0122 & 0.0731 & -0.0211 \\
-0.0369 & -0.0211 & 0.0872
\end{array}\right)$ & 0.2697 \\
\hline  $\hat {\theta}_{\hat w,h_0}$ & $\left(\begin{array}{c}
0.0451\\
 0.0439\\
 0.0291 \end{array}\right)$ & $\left(\begin{array}{ccc}
0.0129 & -0.0034 & -0.0064 \\
-0.0034 & 0.0122 & -0.0054 \\
-0.0064 & -0.0054 & 0.0160
 \end{array}\right)$ & 0.0459\\
\hline  $\hat{\theta}_{\mathrm cox}$ & $\left(\begin{array}{c}
-1.4615\\
 -1.1464\\
 -0.6454 \end{array}\right)$ & $\left(\begin{array}{ccc}
0.0190 & -0.0002 & -0.0011 \\
-0.0002 & 0.0183 & 0.0006 \\
-0.0011 & 0.0006 & 0.0191
 \end{array}\right)$ & 3.9232\\
\hline  
$\hat{\theta}_{\mathrm AFT}$ & $\left(\begin{array}{c}
-5.7300\\
 -6.4569\\
 -8.2406 \end{array}\right)$ & $\left(\begin{array}{ccc}
1.7817 & -0.7809 & -1.3722 \\
-0.7809 & 1.7994 & -0.9062 \\
-1.3722 & -0.9062 & 3.0520
 \end{array}\right)$ & 149.0650\\
\hline
\end{tabular}
%\end{center}
\end{table}

\begin{table}
\caption{Biases, variances and MSE of $\tilde{\theta},\hat {\theta}_{\hat w,h_0},\hat{\theta}_{\mathrm cox}$ and $\hat{\theta}_{\mathrm AFT}$ for $70\%$ of censored data}
\centering
\label{70w}

%\begin{center}
%\fbox{
%{*4{c}}
\begin{tabular}{|l|c|c|c|}
\hline $p=70\%$ & Bias & Variance & MSE \\
\hline \hline $\tilde{\theta}$ & $\left(\begin{array}{c}
0.0957\\
0.0995\\
0.0155 \end{array}\right)$
 & $\left(\begin{array}{ccc}
0.1507 & -0.0067 & -0.1060 \\
-0.0067 & 0.1083 & -0.0305 \\
-0.1060 & -0.0305 & 0.1782
\end{array}\right)$ & $0.4566$ \\
\hline  $\hat {\theta}_{\hat w,h_0}$ & $\left(\begin{array}{c}
0.0460\\
 0.0457\\
 0.0262 \end{array}\right)$ & $\left(\begin{array}{ccc}
0.0179 & -0.0031 & -0.0093 \\
-0.0031 & 0.0141 & -0.0067 \\
-0.0093 & -0.0067 & 0.0210
 \end{array}\right)$ & $0.0579$\\
 \hline  $\hat{\theta}_{\mathrm cox}$ & $\left(\begin{array}{c}
-1.4633\\
 -1.1407\\
 -0.6370 \end{array}\right)$ & $\left(\begin{array}{ccc}
0.0218 & 0.0000 &  -0.0008 \\
0.0000 & 0.0223 & 0.0004\\
-0.0008 & 0.0004 & 0.0211
 \end{array}\right)$ & $3.9132$\\
\hline  
$\hat{\theta}_{\mathrm AFT}$ & $\left(\begin{array}{c}
-5.8857\\
 -6.5324\\
 -7.8799 \end{array}\right)$ & $\left(\begin{array}{ccc}
1.3032 & -0.6241 & -0.9075 \\
-0.6241 & 1.4428 & -0.7425 \\
-0.9075 & -0.7425 & 2.1692
 \end{array}\right)$ & $144.3226$ \\
\hline
\end{tabular}
%}
%\end{center}
\end{table}

We also computed the average weights of $\hat{w}$ for the last four points of $\mathcal I$. For $30\%$ of censoring, we have: $E[\hat{w}(\{0.9\})]=0.685, E[\hat{w}(\{1\})]=0.579, E[\hat{w}(\{1.1\})]=0.586$ and $E[\hat{w}(\{1.2\})]=0.569$ and for $70\%$ of censoring, $E[\hat{w}(\{0.9\})]=0.684,E[\hat{w}(\{1\})]=0.606, E[\hat{w}(\{1.1\})]=0.593$ and $E[\hat{w}(\{1.2\})]=0.538$. Clearly, choosing the measure from the data improves both the bias and the variance of our estimator. Moreover the weights of $\hat{w}(\{k\})$ get smaller for large values of $k$, especially when the proportion of censored data is high. Consequently, the adaptive measure seems to have a significant impact on the quality of the estimation of $\theta_0$. 
Note that even if the Cox estimator has very small variance components, our estimators have significantly smaller biases and MSE than the Cox and AFT estimators. This suggests that the Cox and AFT models do not fit well the data and consequently give really poor estimates of $\theta_0$. Since our model assumptions are less restrictive, this explains why our estimator outperforms the Cox and AFT estimators.

Next, we show how the choice of the parameter $h$ influences the quality of estimation. We consider the fixed measure $w_0$ which puts the same weights $1$ at each point. The bandwidth $\hat{h}$ is chosen adaptively in a regular grid of length $0.05$ in the set $[0.2,1.8].$ The average bandwidths over the $1\,000$ samples were computed and equal to $1.141$ for $30\%$ of censoring and $1.126$ for $70\%$ of censoring.

The performance of the resulting estimator $\hat{\theta}_{w_0,\hat h}$ presented in Table 3 is then compared with the estimators of the previous tables. We observe significant improvement of its MSE compared to $\tilde{\theta}$. As previously, our estimator outperforms $\hat{\theta}_{\mathrm cox}$ and $\hat{\theta}_{\mathrm AFT}$ in term of bias and MSE. We also see that choosing an adaptive bandwidth with a fixed measure or choosing an adaptive measure with a fixed bandwidth leads to a similar quality of estimation of $\theta_0$. However this is no longer true for a high censoring rate: for $70 \%$ of censored data, the MSE of $\hat {\theta}_{\hat w,h_0}$ is almost $2$ times lower than the MSE of $\hat{\theta}_{w_0,\hat h}$.
This shows that the adaptive measure is well suited to the case of censored data: when a large proportion of recurrent events are censored the adaptive measure can compensate the lack of observations due to censoring and allows us to obtain a very accurate estimation of $\theta_0$. 

More simulations results are presented in the Supplementary Material paper with a different setup. We consider a recurrent event process where the number of events in a time interval has a negative binomial distribution. This entails an increase in the variance estimates but other conclusions are similar.\\
\begin{table}[htb]
\caption{Biases, variances and MSE of $\hat {\theta}_{w_0,\hat h}$ for $30\%$ and $70\%$ of censored data}
\label{3070h}
%\begin{center}
\centering
\begin{tabular}{|l|c|c|c|}
\hline  & Bias & Variance & MSE \\
\hline \hline $\hat{\theta}_{w_0,\hat h}$, $p=30\%$& $\left(\begin{array}{c}
0.0405  \\
0.0384\\
0.0393 \end{array}\right)$
 & $\left(\begin{array}{ccc}
0.0143 & -0.009 & -0.0108 \\
-0.009 & 0.0143 & -0.0095 \\
-0.0108 & -0.0095 & 0.0150
\end{array}\right)$ & $0.0483$ \\
\hline  $\hat{\theta}_{w_0,\hat h}$, $p=70\%$ & $\left(\begin{array}{c}
  0.0395\\
  0.0383\\
  0.0362 \end{array}\right)$ & $\left(\begin{array}{ccc}
0.0391& -0.0228& -0.0113\\
-0.0228& 0.0208& -0.0094\\
-0.0113& -0.0094& 0.0357
 \end{array}\right)$ & $0.1099$\\
\hline
\end{tabular}
%\end{center}
\end{table}

\noindent {\bf 5 Conclusion} \label{secconcl}
\setcounter{chapter}{5}
\setcounter{equation}{0} 

We proposed a new procedure to estimate the conditional cumulative mean function of a recurrent event process. We considered both parametric and semiparametric models for the conditional cumulative mean function. Our semiparametric single-index model can be seen as a generalization of both the Cox model and the accelerated failure time model. Moreover, a new feature of our procedure stands in the measure $w$ involved in our estimators which is designed to prevent us from problems in the tail of the distribution due to the presence of censoring. Then, we proposed a data-driven method to choose this measure adaptively. Our criterion is based on the minimization of the mean squared error for the estimation of $\theta_0$ but our procedure is flexible enough to allow the use of any other criteria more adapted to the context. For example, we could consider a criterion directly based on the error of the estimation of $\mu$.

In this work, we mainly focused on kernel estimators for estimating the nonparametric part of our model, providing methods to choose the smoothing parameter from the data. Nevertheless, all our results are valid for any nonparametric estimator of $\mu_{\theta}$ provided it satisfies some convergence properties (see Assumption \ref{aunifconv}).
%still valid for a general class of  and do only rely on the convergence properties . 
Hence, other kinds of estimators may be used provided they satisfy these conditions.\\

\noindent {\bf 6 Appendix} \label{secappend}\\
\noindent {\bf 6.1 Exposition and discussion of assumptions} \label{secsecdisc}
\setcounter{chapter}{6}
\setcounter{section}{0}
\setcounter{equation}{0} 

In this section we state the technical assumptions on which the proofs of the results of Section 3\ref{secasympt} are based. We first present some general properties on a class of functions in order to use empirical processes theory.\\
Let $\mathcal{F}=\{f:(t,z)\in [0,\tau_H]\times \mathcal{Z}\mapsto f(t,z)\}$ be a class of functions with envelope $\bar F$ i.e. such that $|f(t,z)|\leq |\bar F(t,z)|$ for every $(t,z)$ and $f$.  Define, for a probability measure $Q$, the norm $\|\cdot\|_{p,Q}$ as the norm of $L^p(Q)$. The covering number of the class $\mathcal{F}$ for the measure $Q$ denoted by
$N(\varepsilon,\mathcal{F},\|\cdot\|_{p,Q})$ is the smallest number of $L^p(Q)-$balls of radius $\varepsilon$ needed
to cover the set $\mathcal{F}.$ The uniform covering number is defined as $N(\varepsilon,\mathcal{F},\|\cdot\|_{p})=\sup_{Q}N(\varepsilon\|\bar F\|_{p,Q},\mathcal{F},\|\cdot\|_{p,Q})$
where the supremum is taken over all probability measures. 
In what follows, we say that a class of functions $\mathcal{F}$ is a $\|\cdot\|_{p}-VC-$class of functions if there exists two positive constants $\gamma$ and $c$ such that $N(\varepsilon,\mathcal{F},\|\cdot\|_{p})\leq c\varepsilon^{-\gamma}$.
Moreover, a class of functions for which the uniform law of large number holds true is said to be Glivenko Cantelli and a class of functions for which the uniform central limit theorem holds true is said to be Donsker.
We refer the reader to Van der Vaart and Wellner \citeyearpar{Vaart96} for more details on these definitions. 

\noindent A class of functions $\mathcal F$ is said to satisfy one of the following properties if the corresponding condition holds.

\begin{property} \label{avc} 
For a class of functions $\mathcal{F}=\{f:(t,z)\in [0,\tau_H]\times \mathcal{Z}\mapsto f(t,z)\}$
and for any $\tau<\tau_H,$ define
\[\mathcal{F}_{\tau}=\{f(t,\cdot),t\in [0,\tau]\},\]
which is a set of functions defined on $\mathcal{Z}.$ Then, for any $\tau<\tau_H,$ $\mathcal{F}_{\tau}$ is a $VC$-class of functions.
\end{property}

\begin{property}\label{aclasse}
For a class of functions $\mathcal F=\{f:(t,z)\in [0,\tau_H]\times \mathcal{Z}\mapsto f(t,z)\}$, the family of functions defined by $\{(z,y) \mapsto  \int_0^{\tau_H} y(t)f(t,z)dw(t), f\in \mathcal{F}, w\in \mathcal{W}\}$ is Glivenko Cantelli. 
\end{property}

In Section 4 in Supplementary material, we give a general type of sufficient conditions to fulfill these properties. It is easy to check that these technical assumptions are verified when the following conditions hold simultaneously:%6.4.3\ref{secclasses} in the appendix,
\begin{itemize}
\item[-] $\mathcal F$ is a class of polynomial functions $f(t,z)$ (with bounded coefficients),
\item[-] $dE[Y(t)]=g(t) dt$ for some polynomial function $g(t)$,
\item[-] the class of measures is of the form $\mathcal W=\{w:dw(t)=W(t)d\tilde{w}(t)\}$ where $W(t)$ is a decreasing function (of order $t^{-k}$ for $k$ sufficiently high or exponential) and where $\tilde{w}$ belongs to a class of monotone positive uniformly bounded functions sufficiently small (for example, piecewise constant bounded functions with a finite number of jumps).
\end{itemize}

\begin{property}\label{lipsch}
Let $\mathcal F=\{f_{\theta}:(t,z)\in [0,\tau_H]\times \mathcal{Z}\mapsto f_{\theta}(t,z),\theta \in \Theta\}$ be a family of functions indexed by $\theta$. For any $f_{\theta_1},f_{\theta_2}\in \mathcal F$ and $z\in \mathcal{Z},$ we have
\[\sup_{w\in \mathcal{W}}\int_0^{\tau_H} \|f_{\theta_1}(t,z)-f_{\theta_2}(t,z)\|dw(t)\leq c\|\theta_1-\theta_2\|,\]
where $c$ is a positive constant.
\end{property}

\noindent We now introduce the assumptions needed to derive the asymptotic normality of $\hat{\theta}$ in the parametric and semiparametric models. 
%We now introduce the two assumptions needed to obtain our main lemma. Then, we give additional assumptions in order to derive the asymptotic normality of $\hat{\theta}$ in the parametric and semiparametric models.

\bigskip

%\noindent {\textbf{Assumptions for the main lemma.}}
\noindent {\textbf{Assumptions for the parametric model.}}

\noindent In the estimation procedures, we consider integrated versions of the rescaled process with respect to a measure $w$ belonging to a class of measures $\mathcal W$. Detailed comments on this family and its role in the statistical procedure are discussed in Section 3.4\ref{secsecadaptw}.  We need the following assumption for this class of measures.

\begin{assum}  \label{aw} 
Assume there exists some probability measure $w_0$ and a positive constant $c_0$ such that, for any $w\in \mathcal{W},$
\[\int_{t}^{\tau_H} dw(s) \leq c_0 W_0(t),\]
where $W_0(t)=\int_{t}^{\tau_H} dw_0(s)$ can be written as
\[W_0(t)=W_1(t)W_2(t)\] 
where $W_1$ and $W_2$ are two positive and non-increasing functions satisfying
\begin{enumerate}[(1)]
\item $\int_0^{\tau_H} W_1^2(t)(1-F(t-))^{-1}(1-G(t-))^{-2}dG(t)<\infty,$\item $\int_0^{\tau_H} W_2(t)E[dN^*(t)]<\infty,$
\item $\lim_{t\to\tau_H}W_2(t)=0.$
\end{enumerate}
\end{assum}

In particular, Assumption \ref{aw} holds when all the measures $w$ have their support included in a common compact
subspace strictly included in $[0,\tau_H]$. On the other hand, since the function $W_1$ controls $1-\hat{G}(s-)$ in $\hat{Y}(s)$ for $s$
in the vicinity of the tail of the distribution, Assumption \ref{aw} also allows to consider measures $w$ which are
supported in the whole interval $[0,\tau_H]$. Taking $W_1(t)=(1-H(t-))^{1/2}(1-G(t-))^{\varepsilon}$ for some $\varepsilon>0$ would be sufficient to obtain $(1)$. Moreover, in the case where $\tau_H=\infty$, if we suppose that, for $\beta_1>0$, we have $E[N^*(t)]\sim \beta_1 t$ when $t\to\infty$, we could take for example $W_2(t)=t^{-\beta_2}$ for $\beta_2>1$ to fulfill $(2)$ and $(3)$.\\

\noindent We also need the following Hölder condition on the process $N$. This is a technical assumption used in the proof of our main lemma.
\begin{assum}\label{holder}
Suppose there exists $\gamma>0$ such that
\[E\left[\sup_{t\leq \tau,t'\leq \tau}\frac{|N(t)-N(t')|}{|t-t'|^{\gamma}}\right]<\infty.\]
\end{assum}

%A class of functions $\mathcal F$ is said to satisfy Property \ref{avc} if the following condition holds.

%\bigskip
%
%\noindent {\textbf{Assumption for the parametric model.}}

%Before stating Assumption \ref{ainversible1}, we first introduce Properties \ref{aclasse} and \ref{lipsch}.
Let $\nabla_{\theta}\mu_{0}(s,z;\theta_1)$ (resp. $\nabla^2_{\theta}\mu_{0}(s,z;\theta_1)$) denote the vector of partial derivatives (resp. the Hessian matrix) of $\mu_{0}(s,z;\theta)$ with respect to all the components of $\theta$ evaluated at $\theta_1.$ The following assumption can be understood as a regularity assumption on the regression model.
\begin{assum} \label{ainversible1}
Assume that, for all $w\in \mathcal{W},$ the matrix\\
$\Sigma_{w,p}=\int_0^{\tau_H}E[ \nabla_{\theta}\mu_0(t,Z,\theta_0)\nabla_{\theta}\mu_0(t,Z,\theta_0)']dw(t)$
is invertible. Moreover, assume that the classes of functions $\{\mu_0(\cdot,\cdot;\theta), \theta \in \Theta\}$, $\{\nabla_{\theta}\mu_0(\cdot,\cdot;\theta), \theta \in \Theta\}$ and 
$\{\nabla^2_{\theta}\mu_0(\cdot,\cdot;\theta), \theta \in \Theta\}$ satisfy Properties \ref{avc}, \ref{aclasse} and \ref{lipsch}.
\end{assum}

\bigskip

\noindent {\textbf{Additional assumptions for the semiparametric model.}}\\
\noindent The following assumption is similar to Assumption \ref{ainversible1}. Here, $\nabla_{\theta}\mu_{\theta_1}(s,z)$ (resp. $\nabla^2_{\theta}\mu_{\theta_1}(s,z)$) denotes the vector of partial derivatives (resp. the Hessian matrix) of $\mu_{\theta}(s,\theta'z)$ with respect to all the components of $\theta$ evaluated at $\theta_1$. Note that the gradient vector $\nabla_{\theta}\mu_{\theta_1}(s,z)$ does not only depend on $\theta'z$ but also depends on the whole vector $z$. We give an explicit expression of this gradient in Lemma 5, in Supplementary material.%\ref{lemmacalculgradient}.

\begin{assum} \label{ainversible2}
Assume that, for all $w\in \mathcal{W}$, the matrix \\$\Sigma_{w,sp}=\int_0^{\tau_H}E[\nabla_{\theta}\mu_{\theta_0}(t,Z)\nabla_{\theta}\mu_{\theta_0}(t,Z)']dw(t)$
is invertible. Moreover, assume that the classes of functions $\{\mu_{\theta}(\cdot,\cdot), \theta \in \Theta\}$, $\{\nabla_{\theta}\mu_{\theta}(\cdot,\cdot), \theta \in \Theta\}$ and 
$\{\nabla^2_{\theta}\mu_{\theta}(\cdot,\cdot), \theta \in \Theta\}$ satisfy Properties \ref{avc}, \ref{aclasse} and \ref{lipsch}.
\end{assum}

\noindent This assumption is hard to check in practice, since the family of functions $\{\mu_{\theta}:\theta\in \Theta\}$
may have a complex form, which may be impossible to determine explicitly without additional assumptions on the model. Nevertheless, such kind of assumptions are commonplace in the single-index literature and cannot easily be removed. Indeed, they ensure some regularity of the map $\theta\rightarrow \mu_{\theta}(\cdot,\cdot)$ with respect to $\theta.$ Without such regularity assumptions, performing single-index estimation is hopeless, since any error of estimation of $\theta_0$ will be amplified by the irregularity of the map $\theta\rightarrow \mu_{\theta}(\cdot,\cdot).$ In practice, looking at the stability of the estimated functions $\hat{\mu}_{\theta}$ for different values of $\theta$ may inform if this assumption is likely to hold.

In order to enable a data-driven procedure, we need uniform convergence properties for the nonparametric estimators $\hat{\mu}_{\theta}.$
%In the following, $\nabla_{\theta}\mu_{\theta_1}(s,z)$ (resp. $\nabla^2_{\theta}\mu_{\theta_1}(s,z)$) denotes the vector of partial derivatives (resp. the Hessian matrix) of $\mu_{\theta}(s,\theta'z)$ with respect to all components of $\theta$ evaluated at $\theta_1$ \sout{(derivatives with respect to all occurrences of $\theta$)}. Note that the gradient vector $\nabla_{\theta}\mu_{\theta_1}(s,z)$ does not only depend on $\theta'z,$ but depends on the all vector $z.$ We give an explicit expression of this gradient in Lemma \ref{lemmacalculgradient}.
%\newpage
\begin{assum} \label{aunifconv}
Define $\bar{\mu}_{\theta}(t,u)=\sup(\mu_{\theta}(t,u),1)$.
\begin{enumerate}[(1)]
\item  Assume that
\begin{align*}
\sup_{t\leq T_{(n)},\theta\in \Theta, z\in \mathcal{Z}}\left|\frac{\hat{\mu}_{\theta}(t,\theta 'z)-\mu_{\theta}(t,\theta 'z)}{\bar{\mu}_{\theta_0}(t,\theta_0'z)^{\lambda_1+\lambda_2}}\right| & =o_P(1), \\
\sup_{t\leq T_{(n)},\theta\in \Theta, z\in \mathcal{Z}}\left\|\frac{\nabla_{\theta}\hat{\mu}_{\theta}(t,z)-\nabla_{\theta}\mu_{\theta}(t,z)}{\bar{\mu}_{\theta_0}(t,\theta_0'z)^{\lambda_1+\lambda_2}}\right\| & =o_P(1), \\
\sup_{t\leq T_{(n)},\theta\in \Theta, z\in \mathcal{Z}}\left\|\frac{\nabla_{\theta}^2\hat{\mu}_{\theta}(t,z)-\nabla_{\theta}^2\mu_{\theta}(t,z)}{\bar{\mu}_{\theta_0}(t,\theta_0'z)^{\lambda_1+\lambda_2}}\right\| & =o_P(1),
\end{align*}
where $\lambda_1,\lambda_2$ are such that $\lambda_1+\lambda_2\geq 1.$
\item Assume also that
\begin{align*}
\sup_{t\leq T_{(n)}, z\in \mathcal{Z}}\left|\frac{\hat{\mu}_{\theta_0}(t,\theta_0'z)-\mu_{\theta_0}(t,\theta_0'z)}{\bar{\mu}_{\theta_0}(t,\theta_0'z)^{\lambda_1+\lambda_2}}\right| &= O_P(\varepsilon_n), \\
\sup_{t\leq T_{(n)},z\in \mathcal{Z}}\left\|\frac{\nabla_{\theta_0}\hat{\mu}_{\theta_0}(t,z)-\nabla_{\theta}\mu_{\theta_0}(t,z)}{\bar{\mu}_{\theta_0}(t,\theta_0'z)^{\lambda_1+\lambda_2}}\right\| &= O_P(\varepsilon'_n),
\end{align*}
\end{enumerate}
\end{assum}
where $\varepsilon_n\varepsilon_n'=o_P(n^{-1/2}).$

\begin{assum}
Assume that
\[\sup_{z\in\mathcal Z}\int_0^{\tau_H}\mu_{\theta_0}(t,\theta_0'z)^{2(\lambda_1+\lambda_2)}dw(t)<\infty,\]
where $\lambda_1,\lambda_2$ were defined in Assumption \ref{aunifconv}.
\end{assum}

\noindent The following assumption is essential to the empirical processes theory used in our proofs.
We assume that the nonparametric estimators and $\mu_{\theta_0}$ belong to some Donsker classes of functions.

\begin{assum} \label{aregularite}
Assume that there exists some Donsker classes of functions $\mathcal{G}$ and $\mathcal{H}$ such that for all $w\in \mathcal W$
\begin{align*}
(z,y)\longmapsto \int_0^{\tau_H} \big(\mu_{\theta_0}(t,\theta_0'z)-y(t)\big)\nabla_{\theta}\mu_{\theta_0}(t,z)dw(t) & \in \mathcal{G}, \\
z \longmapsto \int_0^{\tau_H} \mu_{\theta_0}(t,\theta_0'z)\nabla_{\theta}\mu_{\theta_0}(t,z)dw(t) & \in \mathcal{H}.
\end{align*}
Moreover, assume that, almost surely for $n$ large enough,
\begin{align*}
(z,y)\longmapsto \int_0^{\tau_H} (\mu_{\theta_0}(t,\theta_0'z)-y(t))\nabla_{\theta}\hat{\mu}_{\theta_0}(t,z)dw(t) & \in \mathcal{G}, \\
z \longmapsto \int_0^{\tau_H} \hat{\mu}_{\theta_0}(t,\theta_0'z)\nabla_{\theta}\mu_{\theta_0}(t,z)dw(t) & \in \mathcal{H}.
\end{align*}
\end{assum}

To give examples of such kind of classes, consider $\mathcal{F}$ and $\mathcal{W}$ as defined in the discussion following Property \ref{aclasse} and suppose, in addition, that the functions $(t,u)\to W_0(t)f(t,u)$ for $f\in \mathcal{F}$ ($f$ is defined on $\mathbb{R}^2$ since $\theta_0'z\in \mathbb{R}$) are twice continuously differentiable with bounded derivatives up to order 2. 

Defining $\mathcal{F}'=\{(u,y)\rightarrow \int_0^{\tau_H}(f_1(t,u)-y(t))f_2(t,u)dw(t),w\in \mathcal{W},f_1,f_2\in \mathcal{F}\},$ it follows from the results of Section 4 in Supplementary material and from the decomposition of the gradient vector $\nabla_{\theta}\mu_{\theta_0}(t,z)$ obtained in Lemma 5 %\ref{lemmacalculgradient}Section 6.4.3\ref{secclasses} 
that  we can decompose
$\int_0^{\tau_H} (\mu_{\theta_0}(t,\theta_0'z)-y(t))\nabla_{\theta}\hat{\mu}_{\theta_0}(t,z)dw(t)=\phi_1(\theta_0'z,y)+z\phi_2(\theta_0'z,y),$
with $\phi_1$ and $\phi_2$ in $\mathcal{F}'.$ A similar decomposition can be used on $\int_0^{\tau_H} \hat{\mu}_{\theta_0}(t,\theta_0'z)\nabla_{\theta}\mu_{\theta_0}(t,z)dw(t).$ Hence, we can consider the class of functions $\mathcal{H}=\mathcal{G}=\mathcal{F}'+z\mathcal{F}'.$\\

\noindent {\bf 6.2 Proof of Lemma \ref{lemmarescale}}  \label{secgrossproof}

Let
\[S^{T_{(n)}}_{n}(f,w)=\frac{1}{n}\sum_{i=1}^n \int_{0}^{T_{(n)}}Y_i(t)f(Z_i,t)dw(t).\]
Write
\begin{align*}
\hat{S}_n(f,w) &= S^{T_{(n)}}_n(f,w)+\frac{1}{n}\sum_{i=1}^n\int_0^{T_{(n)}} f(Z_i,t) \int_0^t \frac{\big(\hat{G}(s-)-G(s-)\big)dN_i(s)}{\big(1-G(s-)\big)\big(1-\hat{G}(s-)\big)}dw(t) \\
&= S^{T_{(n)}}_n(f,w)+R_n(f,w).
\end{align*}
Decompose $f$ into its positive and negative parts denoted respectively by $f^+$ and $f^-$. The expectations of the two resulting sums $S^{T_{(n)}}_n(f^+,w)-S_n(f^+,w)$ and $S^{T_{(n)}}_n(f^-,w)-S_n(f^-,w)$ go to zero faster than $n^{-1/2}$ using Lebesgue's dominated convergence. This entails that
\[\sup_{f\in \mathcal{F},w\in \mathcal{W}} |S^{T_{(n)}}_n(f,w)-S_n(f,w)|=o_P(n^{-1/2}).\]

Let $\tau<\tau_H$ and define $w_{\tau}(t)=w(t)I(t\leq \tau).$ On $[0,\tau],$ we use the asymptotic i.i.d. expansion of the Kaplan-Meier estimator $\hat{G}$ proposed by Gijbels and Veraverbeke~\citeyearpar{Gijbels91} which can also be deduced from Stute~\citeyearpar{Stute95}:
\[\frac{\hat{G}(t)-G(t)}{1-G(t)}=\frac{1}{n}\sum_{j=1}^n \eta_t(T_j,\delta_j)+\tilde{R}_n(t),\]
where $\sup_{t\leq \tau}|\tilde{R}_n(t)|=O_P(n^{-1}\log n)$ and
\[\eta_t(T,\delta)=\frac{(1-\delta)I(T\leq t)}{1-H(T-)}-\int_0^{t}\frac{I(T\geq s)dG(s)}{\big(1-H(s-)\big)\big(1-G(s-)\big)}.\]
Moreover, recall that $\sup_{t\leq \tau}|\hat{G}(t)-G(t)|=O_P(n^{-1/2})$ (see Gill~\citeyearpar{Gill2}, Theorem 2.1) and that $\sup_{t\leq \tau}(1-G(t))(1-\hat{G}(t))^{-1} = O_P(1)$ (see Gill~\citeyearpar{Gill2}, Lemma 2.6). Then, write
\[R_n(f,w_{\tau})=\frac{1}{n^2}\sum_{i,j} \int_0^{T_{(n)}} f(Z_i,t) \int_0^t \frac{\eta_{s-}(T_j,\delta_j)dN_i(s)}{1-G(s-)}dw_{\tau}(t)+R'_n(f,w_{\tau}).\]
Using the fact that $\mathcal{F}$ is an uniformly bounded class, that $\int dw_{\tau}\leq c_0$ from Assumption \ref{aw}
and that $E[N_i(\tau)]< \infty$ for all $\tau,$ we deduce that $\sup_{f,w}|R'_n(f,w_{\tau})|=O_P(n^{-1})$.
The first term in $R_n(f,w_{\tau})$ can be rewritten as
\[\frac{1}{n}\sum_{j=1}^n \int_0^{{T_{(n)}}} \int_0^t \eta_{s-}(T_j,\delta_j)E\big[f(Z,t)d\mu(s|Z)\big]dw_{\tau}(t)+\int \left(\frac{1}{n^2}\sum_{i,j} \psi^{f,t}(Z_i,N_i,T_j,\delta_j)\right)dw_{\tau}(t),\]
where
\[\psi^{f,t}(Z_i,N_i,T_j,\delta_j)=\int_0^t \eta_{s-}(T_j,\delta_j)\left\{\frac{f(Z_i,t)dN_i(s)}{1-G(s-)}-E\big[f(Z,t)d\mu(s|Z)\big]\right\}.\]
Observe that, with probability tending to one, the upper bound $T_{(n)}$ in the integrals can be replaced by $\tau<\tau_H.$
Let $f,f'\in \mathcal{F}$ and $t,t'\in[0,\tau]$. We have
\begin{align}\nonumber
|\psi^{f,t}(Z_i,N_i,T_j,\delta_j)-\psi^{f',t'}(Z_i,N_i,T_j,\delta_j)| & \leq c_{\tau}\bigg(\|f-f'\|_{\infty}N_i(\tau)\\
 & \quad \left.+|t-t'|^{\gamma}\sup_{t,t'\leq \tau}\frac{N_i(t)-N_i(t')}{|t-t'|^{\gamma}}\right),\label{lip}
\end{align}
where $c_{\tau}<\infty$ and $\gamma>0$. Let $\mathcal{H}_{\tau}$ denote the set of all functions $\psi^{f,t}$ when $f$ ranges $\mathcal{F}$ and $t$ ranges $[0,\tau].$
It follows from (\ref{lip}) and Assumption \ref{holder} that $\mathcal{H}_{\tau}$ is a $\|\cdot\|_2-$VC-class of functions.
From this, using the Glivenko-Cantelli property of $\mathcal{H}_{\tau},$
\[\sup_{f,t\leq \tau}\left|\frac{1}{n^2}\sum_{i=1}^n\psi^{f,t}(Z_i,N_i,T_i,\delta_i)\right|=O_P(n^{-1})\]
and 
\[\sup_{f,t\leq \tau}\left|\frac{1}{n^2}\sum_{i\neq j}\psi^{f,t}(Z_i,N_i,T_j,\delta_j)\right|=O_P(n^{-1}),\]
since this can be seen as the supremum of a second order degenerate $U-$process indexed by $\mathcal{H}_{\tau}$ (see Sherman~\citeyearpar{Sherman94}).
This leads to the i.i.d. representation for $\hat{S}_n(f,w_{\tau})$ for any $\tau<\tau_H.$

\noindent Similarly, write
\[\hat{S}_{n}(\hat{f},w_{\tau})=S^{T_{(n)}}_n(\hat{f},w_{\tau})+R_n(\hat{f}-f,w_{\tau})+R_{n}(f,w_{\tau})\]
and using the fact that $\sup_{f\in \mathcal{F}}\|f-\hat{f}\|_{\infty}=o_P(1)$
and that $\sup_{t\leq \tau}|\hat{G}(t)-G(t)|=O_P(n^{-1/2})$, we deduce that
$\sup_{f,w}|R_n(\hat{f}-f,w_{\tau})|=o_P(n^{-1/2})$. The representation for $\hat{S}_{n}(\hat{f},w_{\tau})$ follows.

Now, we make $\tau$ tend to $\tau_H.$ Let
$\hat{P}_{n}(f,w)=\hat{S}_{n}(f,w)-S^{T_{(n)}}_{n}(f,w)$ and
$P^{\tau}_{n}(f,w)=\hat{S}_n(f,w_{\tau})-S^{T_{(n)}}_n(f,w_{\tau}).$ Since the class $\mathcal{F}$
is uniformly bounded, we get
\begin{align*}
|\hat{P}_{n}(f,w)-P^{\tau}_{n}(f,w)| &\leq \frac{M}{n}\sum_{i=1}^n \int_{\tau}^{T_{(n)}}\!\!\int_0^t\frac{|\hat{G}(s-)-G(s-)|}{\big(1-G(s-)\big)\big(1-\hat{G}(s-)\big)} dN_i(s)dw(t)\\
& \leq \frac{M'}{n} \sum_{i=1}^n\int_{0}^{T_{(n)}} \frac{W_0(s\vee\tau)|\hat{G}(s-)-G(s-)|dN_i(s)}{\big(1-G(s-)\big)\big(1-\hat{G}(s-)\big)},
\end{align*}
(with $a\vee b$ denoting the maximum between $a$ and $b$) where the last inequality is obtained from Fubini's theorem and Assumption \ref{aw}.
%$$|\hat{P}^{\tau}_{n}(f,w)-P^{\tau}_{n}(f,w)|\leq M \int_{\tau}^{T_{(n)}} \frac{W_0(s)|\hat{G}(s-)-G(s-)|dN_i(s)}{[1-G(s-)][1-\hat{G}(s-)]}.$$
From Theorem 1.2 in Gill~\citeyearpar{Gill2}, Assumption \ref{aw} and the fact that $\sup_{t\leq T_{(n)}}(1-G(t-))(1-\hat{G}(t-))^{-1}=O_P(1)$ (see again Gill~\citeyearpar{Gill2}, we get that
\[|\hat{P}_{n}(f,w)-P^{\tau}_{n}(f,w)|\leq \frac{A_n}{n}\sum_{i=1}^n \int_{0}^{T_{(n)}} \frac{W_2(s\vee\tau)dN_i(s)}{1-G(s-)},\]
%$$|\hat{P}^{\tau}_{n}(f,w)-P^{\tau}_{n}(f,w)|\leq A_n \frac{1}{n}\sum_{i=1}^n \int_{\tau}^{\tau_H} \frac{W_2(s)dN_i(s)}{1-G(s-)},$$
where $A_n=O_P(n^{-1/2}).$ The result follows from Lemma 6 in Supplementary material.\\%\ref{lemmatightness}.\\

\noindent {\bf 6.3 Assumptions for the uniform convergence of the nonparametric estimators} \label{secpreuvenonpara}

In the supplementary material, we show that the kernel estimator $\hat{\mu}_{\theta,h}$ defined by (\ref{kernel}) satisfies the convergence rates required by Assumption \ref{aunifconv}. This requires the following assumption on the kernel and the bandwidth.
\begin{assum}\label{kern}
Assume that
\begin{enumerate}[(1)]
\item $K$ has a compact support, say $[-1,1]$, $\int_{\mathbb{R}}K(s)ds =1$ and $\sup_x|K(x)|<\infty$,
\item $K$ is a twice differentiable and second-order kernel with derivatives of order 0, 1 and 2 of bounded variation,
\item $\mathcal{K}:=\{K\big((x-\cdot)/h\big):h>0, x\in \mathbb{R}^d\}$ is a pointwise measurable class of functions,
\item $h\in \mathcal{H}_n\subset [an^{-\alpha},bn^{-\alpha}]$ with $a,b>0$
and $\alpha \in (1/8,1/5).$
\end{enumerate}
\end{assum}

From the definition of our estimator, problems arise when studying $\hat{\mu}_{\theta,h}$ for $t$ in the tail of the distribution. This is a common problem when studying Kaplan-Meier estimators but it can be circumvented by some moment conditions on the response and censoring distribution. For instance, in the classical censored framework, Stute~\citeyearpar{Stute95} used the function $C_G$ to compensate the bad behavior of the Kaplan Meier estimator in the tail of the distribution.
Therefore we also require the following assumption, which gives a similar moment condition but adapted to our recurrent event context. 

\begin{assum}\label{cens}
Assume that, for some $\varepsilon>0$,
\[\sup_{t,u}\frac{C_G(t)^{7/20+\varepsilon}}{\bar{\mu}_{\theta_0}(t,u)^{\lambda_1}}<\infty\]
and
\[\sup_{t,u}\frac{\int_0^t\big(1-G(s-)\big)E[N^*(s)dN^*(s)]}{\big(1-G(t-)\big)^2\bar{\mu}_{\theta_0}(t,u)^{2\lambda_2}}<\infty,\]
where $\lambda_1$ and $\lambda_2$ are defined in Assumption \ref{aw}.
\end{assum}

These conditions allow us to consider a process $N^*$ and variables $D$ and $C$ that are supported on the whole interval $[0,\tau_H]$. 
However they will hold true only if there is enough information on the recurrent event process in the tails of the distribution. 
For further illustration take $\mu^k_{\theta_0}(t,u)\sim c_k(1-G(t))^{-\beta_1}$, for $k=1$ and $2$, for $t$ in a neighborhood of $\tau_H$, $u \to \infty$ and where $c_1$, $c_2$, $\beta_1$ are three positive constants. Take also, for $c_3>0$ and $\beta_2>0$, $1-F(t)\sim c_3(1-G(t))^{\beta_2}$ for $t$ in a neighborhood of $\tau_H$. Then it can be shown that these conditions are verified for example in the case $\beta_1>1$, $\beta_2=1$ and $\lambda_1=\lambda_2=1$.
\par
%\noindent {\large\bf Acknowledgment}

%Write the acknowledgment here.
%\par
\bibliographystyle{chicago}
\bibliography{BGLbibli}

%\noindent{\large\bf References}
%\begin{description}
%\item
%Andrews, D. W. K. (1984). Non-strong mixing autoregressive processes.
%   {\it J. Appl. Probab.} {\bf 21}, 930-934.
%\item
%Avram, F. and Taqqu, M. S. (1987)  Noncentral limit theorems and Appell
%polynomials. {\it Ann. Probab.} {\bf 15}, 767-775.
%\item
%Bradley, R. C. (1986). Basic properties of strong mixing conditions. In {\it
%Dependence
%in Probability and Statistics} (Edited by E. Eberlein and M. S. Taqqu),
%162-192. Birkh\"auser, Boston.
%\item
%Fox, R. and Taqqu, M. (1987). Central limit theorems for quadratic forms in
%random
%variables having long-range dependence. {\it Probab. Theory Related Fields}
%{\bf 74}, 213-240.
%\end{description}
% \vskip .65cm
% \noindent
% Olivier Bouaziz, MAP5, Universit\'e Paris Descartes
% \vskip 2pt
% \noindent
% E-mail: (olivier.bouaziz@parisdescartes.fr)
% \vskip 2pt
% \noindent
% S\'egolen Geffray, IRMA, Universit\'e de Strasbourg
% \vskip 2pt
% \noindent
% E-mail: (geffray@math.unistra.fr)
% \vskip 2pt
% \noindent
% Olivier Lopez, LSTA, Universit\'e Paris VI
% \vskip 2pt
% \noindent
% E-mail: (olivier.lopez0@upmc.fr)
% \vskip .3cm
\end{document}

\renewcommand{\baselinestretch}{1.2}
%\lhead[\fancyplain{} \leftmark]{}
%\chead[]{}
%\rhead[]{\fancyplain{}\rightmark}
%\cfoot{}
%\headrulewidth=0pt

\title{{\itshape Supplementary material for \\Semiparametric inference for the recurrent events process by means of a single-index model}}
\author{Olivier Bouaziz$^{a}$$^{\ast}$\thanks{Corresponding author$^\ast$. Email: olivier.bouaziz@parisdescartes.fr \vspace{6pt}}, Ségolen Geffray$^{b}$ and Olivier Lopez$^{c,d}$\\\vspace{6pt} $^{a}${\em{MAP5, Universit\'e Paris Descartes, 45 rue des Saints P\`eres, 75270 Paris Cedex 06}}; $^{b}${\em{IRMA, Universit\'e de Strasbourg, 7 rue Ren\'e-Descartes 67084 Strasbourg Cedex}}; $^{c}${\em{ENSAE Paris-Tech \& CREST, Laboratoire de Finance et d'Assurance, 3 avenue Pierre Larousse, 92245 Malakoff Cedex France}}; $^{d}$ {\em{Sorbonne Universités, UPMC Université Paris VI, EA 3124, LSTA, 4 place Jussieu 75005 Paris}}\\}
\maketitle

\section{Extended simulation study}\label{binomneg}

In this section, we extend the simulation study performed in Section 4 of the main article. In order to take into account heterogeneity among individuals, we consider a gamma frailty model where $G_i\sim\Gamma(a,1/a)$, $i=1,\ldots,n$, is a gamma variable with shape $a$ and scale $1/a$. The process $\tilde N_i(\cdot)$ is simulated in such a way that, conditionally on $G_i$ and $Z_i$, $\tilde N_i(\cdot)$ is a Poisson process with intensity $(\theta_0'Z_i+5)G_i$. This ensures us that, marginally, $\tilde N_i(t) | Z_i$ has a negative binomial distribution with mean equal to 
$(\theta_0'Z_i+5)t$ and variance equal to $(\theta_0'Z_i+5)t+(\theta_0'Z_i+5)^2t^2/a$. Note also that the process of interest $N^*(\cdot)$ has a conditional expectation equal to% the cumulative mean function
\[E[N_i^*(t)|Z_i]=(\theta_0'Z_i+5)\int_0^t (1-F(s-))ds , \quad i=1, \ldots,n,\]
which is the same as in the simulations section of the main paper, but has a larger variance than its expectation. 

We put $a=2$. The distributions of the variables $D_i$ and $C_i$, the parameters and the family of weights are all set to the same values as in the main paper. In Tables~\ref{30w} and~\ref{70w} we report the results of our estimators $\tilde \theta$ and $\hat \theta_{\hat w,h_0}$ over $1\,000$ simulations of samples of size $100$ for two rates of censoring ($30\%$ and $70\%$). We also compare these results with the Cox estimator as previously.
The average weights of $\hat w$ were also computed. For $30\%$ of censoring, we obtain, $E[\hat{w}(\{0.9\})]=0.711, E[\hat{w}(\{1\})]=0.563, E[\hat{w}(\{1.1\})]=0.420$ and $E[\hat{w}(\{1.2\})]=0.337$ and for $70\%$ of censoring, $E[\hat{w}(\{0.9\})]=0.725,E[\hat{w}(\{1\})]=0.561, E[\hat{w}(\{1.1\})]=0.422$ and $E[\hat{w}(\{1.2\})]=0.329$. 

When no weights are used in the estimation procedure, the simulation results are very similar to those obtained in the main paper. However, an increase in the variance estimates can be seen in the negative binomial context compared to the Poisson framework. %Similarly This is no surprising since 
As previously, the adaptive measure seems to play an important role in the estimation performance of $\theta_0$. However, the improvement in the quality of estimation in the negative binomial context is clearly not as remarkable as in the Poisson context. Finally, our estimators still outperform the Cox estimator. The latter is still biased and has also a greater variance than in the Poisson situation. 

All these results emphasize the fact that the recurrent events, in this simulation design, have a greater variance than in the Poisson context. This seems to deteriorate the quality of estimation of all estimators. However, the adaptive choice of the weights can still improve greatly the simulation results in term of MSE, especially in the case of high censoring, where the MSE is almost divided by $3$ (and divided by $1.5$ for $30\%$ of censored data). %A different choice of the family of measures could probably 

\begin{table}[h]
\caption{Biases, variances and MSE of $\tilde{\theta},\hat {\theta}_{\hat w,h_0}$ and $\hat{\theta}_{\mathrm cox}$ for $30\%$ of censored data}% and $\hat{\theta}_{\mathrm AFT}$
\label{30w}
%\begin{center}
\centering
\begin{tabular}{|l|c|c|c|}
\hline $p=30\%$ & Bias & Variance & MSE \\
\hline \hline $\tilde{\theta}$ & $\left(\begin{array}{c}
0.0827 \\
0.0606\\
0.0552 \end{array}\right)$
 & $\left(\begin{array}{ccc}
0.1008 &  -0.0176 & -0.0551 \\
-0.0176 & 0.0794 & -0.0433 \\
-0.0551 & -0.0433 & 0.0980
\end{array}\right)$ & 0.2880 \\
\hline  $\hat {\theta}_{\hat w,h_0}$ & $\left(\begin{array}{c}
0.0634\\
0.0597\\
0.0429 \end{array}\right)$ & $\left(\begin{array}{ccc}
0.0559 & -0.0202 & -0.0242 \\
-0.0202 & 0.0624 & -0.0227 \\
-0.0242 & -0.0227 & 0.0679
 \end{array}\right)$ & 0.1956\\
\hline  $\hat{\theta}_{\mathrm cox}$ & $\left(\begin{array}{c}
-1.4975\\
-1.1696\\
-0.6608 \end{array}\right)$ & $\left(\begin{array}{ccc}
0.0626 & -0.0002 & 0.0011 \\
-0.0002 & 0.0650 & 0.0056 \\
0.0011 & 0.0056 & 0.0607
 \end{array}\right)$ & 4.2353\\
%\hline  
%$\hat{\theta}_{\mathrm AFT}$ & $\left(\begin{array}{c}
%-5.7300\\
% -6.4569\\
% -8.2406 \end{array}\right)$ & $\left(\begin{array}{ccc}
%1.7817 & -0.7809 & -1.3722 \\
%-0.7809 & 1.7994 & -0.9062 \\
%-1.3722 & -0.9062 & 3.0520
% \end{array}\right)$ & 149.0650\\
\hline
\end{tabular}
%\end{center}
\end{table}

\begin{table}
\caption{Biases, variances and MSE of $\tilde{\theta},\hat {\theta}_{\hat w,h_0}$ and $\hat{\theta}_{\mathrm cox}$ for $70\%$ of censored data}% and $\hat{\theta}_{\mathrm AFT}$
\centering
\label{70w}

%\begin{center}
%\fbox{
%{*4{c}}
\begin{tabular}{|l|c|c|c|}
\hline $p=70\%$ & Bias & Variance & MSE \\
\hline \hline $\tilde{\theta}$ & $\left(\begin{array}{c}
0.0913\\
0.0748\\
0.0578 \end{array}\right)$
 & $\left(\begin{array}{ccc}
0.1449 & -0.0312 & 0.0287 \\
-0.0412 & 0.1210 & -0.0265 \\
0.0287 & -0.0265 & 0.1927
\end{array}\right)$ & $0.6417$ \\
\hline  $\hat {\theta}_{\hat w,h_0}$ & $\left(\begin{array}{c}
0.0740\\
0.0624\\
0.0411 \end{array}\right)$ & $\left(\begin{array}{ccc}
0.0643 & -0.0242 & -0.0244 \\
-0.0242 & 0.0682 & -0.0242 \\
-0.0244 & -0.0242 & 0.0731
 \end{array}\right)$ & $0.2167$\\
 \hline  $\hat{\theta}_{\mathrm cox}$ & $\left(\begin{array}{c}
-1.5005\\
-1.1744\\
-0.6449 \end{array}\right)$ & $\left(\begin{array}{ccc}
0.0724 & -0.0004 & 0.0002 \\
-0.0004 & 0.07156 & 0.0771\\
0.0002 & 0.0771 & 0.0771
 \end{array}\right)$ & $4.2676$\\
%\hline  
%$\hat{\theta}_{\mathrm AFT}$ & $\left(\begin{array}{c}
%-5.8857\\
% -6.5324\\
% -7.8799 \end{array}\right)$ & $\left(\begin{array}{ccc}
%1.3032 & -0.6241 & -0.9075 \\
%-0.6241 & 1.4428 & -0.7425 \\
%-0.9075 & -0.7425 & 2.1692
% \end{array}\right)$ & $144.3226$ \\
\hline
\end{tabular}
%}
%\end{center}
\end{table}

\section{Uniform convergence of the nonparametric estimators}\label{secpreuvenonpara}
%\noindent {\bf 6.3 Uniform convergence of the nonparametric estimators} 

In this section, we show that the kernel estimator $\hat{\mu}_{\theta,h}$ defined by (2.10) satisfies the convergence rates required by Assumption 7, under Assumptions 10 and 11. Introduce the quantity
\[\tilde{\mu}_{\theta,h}(t,u)=\sum_{i=1}^n\int_{0}^t \frac{K\left(\frac{\theta 'Z_i-u}{h}\right)dN_i(s)}{\sum_{j=1}^n K\left(\frac{\theta 'Z_j-u}{h}\right)\big(1-G(s-)\big)}.\]
We first study the convergence rate of the difference between $\tilde{\mu}_{\theta,h}$ and $\mu_{\theta}$ and their derivatives. Since no Kaplan-Meier functions are involved in this expression, we can use classical results on uniform convergence of kernel estimators, mainly from Einmahl and Mason~\citeyearpar{Einmahl05}.

We also introduce a trimming function. 
Its purpose is to circumvent problems caused by too small values of the denominator in the definition of $\hat{\mu}_{\theta,h}$. 
Indeed, to ensure uniform consistency of our estimator, we need to bound this denominator away from zero. We use the same methodology as in Delecroix {\it et al.}~\citeyearpar{Delecroix06}. 
Let $f_{\theta_0'Z}$ denote the density of $\theta_0'Z$ and define the ``ideal'' trimming function $J_{\theta_0}(\theta_0'Z,c)=I(\theta_0'Z \in B_0)$ where $B_0=\{u:f_{\theta_0'Z}(u)\geq c\}$ for some constant $c>0$. As in Delecroix {\it et al.}~\citeyearpar{Delecroix06} (see also Lopez~\citeyearpar{Lopez09}), we first assume that we know some set $B$ on which $\inf\{f_{\theta'Z}(\theta'z): z\in B, \theta \in \Theta\}>c$ where $c$ is a strictly positive constant. 
In a preliminary step, we can use this set $B$ to compute the preliminary trimming $J_B(z)=I(z\in B).$ 
Using this trimming function and a deterministic sequence of bandwidth $h_0$ satisfying $(4)$ in Assumption 10 we define a preliminary estimator $\theta_n$ of $\theta_0$ as
\[\theta_n(w)=\argmin_{\theta\in \Theta}M_{n,w}(\theta,\hat{\mu}_{\theta})J_B(z).\]
Given this preliminary consistent estimator of $\theta_0$, we use the following trimming $J_n(\theta_n'Z,c)=I(\hat{f}_{\theta_n'Z}(\theta_n'Z)\geq c)$ which appears to be asymptotically equivalent to $J_{\theta_0}(\theta_0'Z,c)$ (see e.g. Lopez~\citeyearpar{Lopez09}). Then, our final estimator consists of
\[\hat{\theta}(w)=\argmin_{\theta\in \Theta_n}M_{n,w}(\theta,\hat{\mu}_{\theta})J_n(\theta_n'z,c),\]
where $\Theta_n$ is a shrinking neighborhood of $\theta_0$ accordingly to our preliminary estimator $\theta_n$.\\

\noindent As announced, the next proposition gives the rates of convergence of $\tilde{\mu}_{\theta,h}$ and its derivatives. 
Since we need a convergence over $\theta\in \Theta$, the trimming we need to use is $J_{\theta}(\theta 'Z,c):=I(\hat{f}_{\theta 'Z}(\theta 'Z)\geq c)$. 
But notice that $J_{\theta_0}(\theta_0'Z,c)$ can be replaced by $J_{\theta}(\theta 'Z,c/2)$ on shrinking neighborhoods of $\theta_0$.

\begin{proposition}\label{rate}
Under Assumption 10, for $z$ such that $J_{\theta}(\theta 'z,c)=1$ almost surely, we have
\begin{align}
\sup_{t\leq T_{(n)},\theta, z, h}\sqrt{\frac{nh}{\log n}}\left|\frac{\tilde{\mu}_{\theta}(t,\theta 'z)-\mu_{\theta}(t,\theta 'z)}{\bar{\mu}_{\theta_0}(t,\theta_0'z)^{\lambda_1+\lambda_2}}\right| & = O_{P}\left(1\right),\label{kern1} \\
\sup_{t\leq T_{(n)},\theta, z, h}\sqrt{\frac{nh^{3}}{\log n}}\left\|\frac{\nabla_{\theta}\tilde{\mu}_{\theta}(t,z)-\nabla_{\theta}\mu_{\theta}
(t,z)}{\bar{\mu}_{\theta_0}(t,\theta_0'z)^{\lambda_1+\lambda_2}}\right\|
& =O_{P}\left(1\right),\label{kern2}\\
\sup_{t\leq T_{(n)},\theta, z, h}\sqrt{\frac{nh^5}{\log n}}\left\|\frac{\nabla_{\theta}^2\tilde{\mu}_{\theta}(t,z)-\nabla_{\theta}^2
\mu_{\theta}(t,z)}{\bar{\mu}_{\theta_0}(t,\theta_0'z)^{\lambda_1+\lambda_2}}\right\|& =O_{P}\left(1\right)\label{kern3}.
\end{align}
\end{proposition}
\begin{proof}
The proofs of (\ref{kern1})-(\ref{kern3}) are all similar. The most delicate term to handle, coming from (\ref{kern3}), is
\[\hat{A}_{\theta}^{n,h}(t,z):=\frac{1}{nh^3}\sum_{i=1}^n\frac{(Z_i-z)^2}{\bar{\mu}_{\theta_0}(t,\theta_0'z)^{\lambda_1+\lambda_2}}K''\left(\frac{\theta'Z_i-\theta'z}{h}\right)
\int_0^t\frac{dN_i(s)}{1-G(s-)}.\]
Consider the class of functions $\mathcal K$ introduced in Assumption 10. From Nolan and Pollard~\citeyearpar{Nolan87}, it can easily be seen that, using a kernel $K$ satis\-fying Assumption 10, for some $c'>0$ and $\nu>0$, we have $N(\varepsilon, \mathcal K, \|\cdot\|_{\infty})\leq c'\varepsilon^{-\nu}, 0<\varepsilon<1$.

%We first recall some classical properties on kernel estimators. Consider the class of functions $\mathcal K$ introduced in Assumption \ref{kern}. Let $N(\varepsilon, \mathcal K, d_Q)$ be the minimal number of balls $\{g:d_Q(g,g')<\varepsilon\}$ of $d_Q$-radius $\varepsilon$ needed to cover $\mathcal K$. For $\varepsilon>0$, let $N(\varepsilon,\mathcal K)=\sup_QN(\kappa\varepsilon,\mathcal{K},d_Q)$, where the supremum is taken over all probability measures $Q$ on $(\mathbb{R}^d,\mathcal B)$, $d_Q$ being the $L_2(Q)$-metric. >From Nolan and Pollard~\citeyearpar{Nolan87}, it can easily be seen that, using a kernel $K$ satis\-fying Assumption \ref{kern}, for some $C>0$ and $\nu>0$, $N(\varepsilon, \mathcal K)\leq C\varepsilon^{-\nu}, 0<\varepsilon<1$. 
Then, concerning the uniformity with respect to $\theta$, Lemma 22 (ii) of Nolan and Polard~\citeyearpar{Nolan87} shows that the family of functions $\left\{(Z,N)\longmapsto \hat{A}_{\theta}^{n,h}(t,z)\right\}$
%\frac{1}{nh^3}\sum_{i=1}^n(Z_i-z)^2K''\left(\frac{\theta'Z_i-\theta'z}{h}\right)
%\int_0^t\frac{dN_i(s)}{1-G(s-)}\right\}
satisfies the assumptions of Proposition 1 in Einmahl and Mason~\citeyearpar{Einmahl05}. 

\noindent Define 
\[\tilde{A}_{\theta}^{h}(t,z):=\frac{1}{h^3}E\left[\frac{(Z-z)^2}{\bar{\mu}_{\theta_0}(t,\theta_0'z)^{\lambda_1+\lambda_2}}K''\left(\frac{\theta 'Z-\theta 'z}{h}\right)
\int_0^t\frac{dN(s)}{1-G(s-)}\right],\]
\[A_{\theta}^{h}(t,z):=\left.\frac{\partial^2}{\partial u^2}\Bigg\{E\,\left[\frac{(Z-z)^2}{\bar{\mu}_{\theta_0}(t,\theta_0'z)^{\lambda_1+\lambda_2}}\int_0^t\frac{dN(s)}{1-G(s-)}\bigg|\theta 'Z=u\right]f_{\theta 'Z}(u)\Bigg\}\right|_{u=\theta 'z}\]
and apply Talagrand's inequality (see Talagrand~\citeyearpar{Talagrand94}, see also Einmahl and Mason~\citeyearpar{Einmahl05}) to obtain that
\[\sup_{t\leq T_{(n)},\theta, z, h}|\hat{A}_{n,h}(t,z)-\tilde{A}_{n,h}(t,z)|=O_{P}\left(n^{-1/2}h^{-5/2}(\log n)^{1/2}\right).\]
For the bias term, classical kernel arguments (see for instance Bosq and Lecoutre~\citeyearpar{Bosq97}) show that
\[\sup_{t\leq T_{(n)},\theta, z, h}|\tilde{A}_{n,h}(t,z)-A_{n,h}(t,z)|=O(h^2).\]  
\end{proof}

It remains to study $\hat{\mu}_{\theta,h}-\tilde{\mu}_{\theta,h}$. 
The following lemma gives some precision on the difference between the Kaplan Meier weights of $\hat{\mu}_{\theta,h}$ and the ``ideal'' weights involving the true function $G$ in $\tilde{\mu}_{\theta,h}$. 

\begin{lemma}\label{jump}
Let $\hat{\Lambda}_G(s)=(1-\hat{G}(s-))^{-1}$, $\tilde{\Lambda}_G(s)=(1-G(s-))^{-1}$ and 
\[C_G(t)=\int_0^t\frac{dG(s)}{\big(1-G(s-)\big)\big(1-H(s-)\big)}.\]
\begin{enumerate}[(1)]
\item We have
\[\sup_{t\leq T_{(n)}}\frac{1-G(t)}{1-\hat{G}(t)}=O_P(1).\]
\item For all $0\leq \beta\leq 1$ and $\varepsilon>0$, we have
\[|\hat{\Lambda}_G(s)-\tilde{\Lambda}_G(s)|\leq R_n(s)\tilde{\Lambda}_G(s)C_G(s)^{\beta(1/2+\varepsilon)},\]
where $\sup_{s\leq T_{(n)}}R_n(s)=O_P(n^{-\beta/2}).$ 
\end{enumerate}
\end{lemma}

\begin{proof}
\begin{enumerate}[(1)]
\item This result is a consequence of Lemma 2.6 in Gill~\citeyearpar{Gill2}. 
\item For $0\leq \beta\leq 1$ and $\varepsilon>0$, write
\[\hat{\Lambda}_G(s)-\tilde{\Lambda}_G(s)=\tilde{\Lambda}_G(s)C_G(s)^{\beta(1/2+\varepsilon)}
\big(R_G(s)C_G(s)^{-1/2-\varepsilon}\big)^{\beta}
\big(R_G(s)\big)^{1-\beta}\frac{1-G(s-)}{1-\hat{G}(s-)},\]
where $R_G(s)=\big(\hat{G}(s-)-G(s-)\big)\big(1-G(s-)\big)^{-1}.$ Since $\int_0^{\tau_H}C_G(s)^{-1-2\varepsilon}dC_G(s)<\infty$, apply Theorem 1 in Gill~\citeyearpar{Gill2} and use the first part of the current lemma to conclude the proof.  
\end{enumerate}
\end{proof}

The next proposition gives the convergence rate of $\hat{\mu}_{\theta,h}-\tilde{\mu}_{\theta,h}$. Notice that if $w$ is supported on a compact interval, we only need this result on a compact subset of $[0,T_{(n)}]$ and in this case Assumption 11 is automatically fulfilled.  

\begin{proposition}\label{rateKM}
Under Assumptions 10 and 11, for $z$ such that $J_{\theta}(\theta'z,c)=1$ almost surely, we have
\begin{align}
\sup_{t\leq T_{(n)},\theta, z, h}\left|\frac{\hat{\mu}_{\theta}(t,\theta 'z)-\tilde{\mu}_{\theta}(t,\theta 'z)}{\bar{\mu}_{\theta_0}(t,\theta_0'z)^{\lambda_1+\lambda_2}}\right| & = O_{P}\left(n^{-7/20}\right), \label{mu}\\
\sup_{t\leq T_{(n)},\theta, z, h}h\left\|\frac{\nabla_{\theta}\hat{\mu}_{\theta}(t,z)-\nabla_{\theta}\tilde{\mu}_{\theta}(t,z)}{\bar{\mu}_{\theta_0}(t,\theta_0'z)^{\lambda_1+\lambda_2}}\right\| & =O_{P}\left(n^{-7/20}\right),\label{gradient}\\
\sup_{t\leq T_{(n)},\theta, z, h}h^2\left\|\frac{\nabla_{\theta}^2\hat{\mu}_{\theta}(t,z)-\nabla_{\theta}^2\tilde{\mu}_{\theta}(t,z)}{\bar{\mu}_{\theta_0}(t,\theta_0'z)^{\lambda_1+\lambda_2}}\right\| & =O_{P}\left(n^{-7/20}\right).\label{hessian}
\end{align}
\end{proposition}

\begin{proof}
We only prove (\ref{hessian}) since (\ref{mu}) and (\ref{gradient}) can be handled similarly. Let us consider the following term involving the second derivative of $K$
%A_{n,h}(t,z):=
\[\frac{1}{nh^3}\sum_{i=1}^n(Z_i-z)^2K''\left(\frac{\theta 'Z_i-\theta 'z}{h}\right)\left(\bar{\mu}_{\theta_0} (t,\theta_0'z)^{\lambda_1+\lambda_2}f_{\theta 'Z}(\theta 'z)\right)^{-1}\int_0^t\big(\hat{\Lambda}(s)-\tilde{\Lambda}(s)\big)dN_i(s).\]
From Lemma \ref{jump}, this term can be bounded by  
%\[\sup_{t\leq T_{(n)},\theta, z, h}|A_{n,h}(t,z)|\leq 
\begin{equation}
O_P(n^{-\beta/2}h^{-2})\left|\frac{1}{nh}\sum_{i=1}^nK''\left(\frac{\theta 'Z_i-\theta 'z}{h}\right)\bar{\mu}_{\theta_0}(t,\theta_0'z)^{-(\lambda_1+\lambda_2)}\int_0^t\tilde{\Lambda}(s)C_G(s)^{\beta(1/2+\varepsilon)}dN_i(s)\right|\label{euclid}
\end{equation}
where the $O_P-$ rate does not depend on $t, \theta, z$ nor $h$. Now, consider the family of functions indexed by $t, \theta, z$ and $h$, 
\[\left\{(Z,N)\longmapsto K''\left(\frac{\theta 'Z-\theta 'z}{h}\right)\bar{\mu}_{\theta_0}(t,\theta_0'z)^{-(\lambda_1+\lambda_2)}\int_0^t\tilde{\Lambda}(s)C_G(s)^{\beta(1/2+\varepsilon)}dN(s)\right\}.\]
This family is Euclidian (see Nolan and Pollard~\citeyearpar{Nolan87}) for an envelope 
\[\sup_{t,z}\frac{\tilde{\Lambda}(t)C_G^{\beta(1/2+\varepsilon)}(t)N(t)}{\bar{\mu}_{\theta_0}(t,\theta_0'z)^{\lambda_1+\lambda_2}}\]
which is, for $\beta=7/10$, square integrable from Assumption 11. Then, using the results of Sherman~\citeyearpar{Sherman94}, the second part of (\ref{euclid}) is $O_P(1)$ uniformly in $t, \theta, z$ and $h$.\\
\end{proof}

\noindent Finally, combination of Propositions \ref{rate} and \ref{rateKM} leads to the following result. 

\begin{corollary}
Under Assumptions 10 and 11, for $z$ such that $J_{\theta}(\theta'z,c)=1$ almost surely,
\[\sup_{t\leq T_{(n)},\theta, z, h}\frac{|\hat{\mu}_{\theta}(t,\theta 'z)-\mu_{\theta}(t,\theta 'z)|\cdot\|\nabla_{\theta}\hat{\mu}_{\theta}(t,z)-\nabla_{\theta}\mu_{\theta}(t,z)\|}{|\bar{\mu}_{\theta_0}(t,\theta_0'z)|^{2(\lambda_1+\lambda_2)}} =o_P(n^{-1/2}).\]
\\
\end{corollary}

\section{Technical lemmas} \label{sectec}
\subsection{Gradient vector in the single-index model}
%\noindent {\bf 6.4 Technical lemmas} \label{sectec}\\
%\noindent {\bf 6.4.1 Gradient vector in the single-index model}

\begin{lemma} \label{lemmacalculgradient}
Let $\mu '_{\theta_0}(t|u)=\frac{\partial}{\partial u} \mu_{\theta_0}(t,u)$ (assuming that $\mu_{\theta_0}(\cdot,\cdot)$ is $C^1$).
Then, for every $(t,z),$ the map $\theta\mapsto \mu_{\theta}(t,\theta 'z)$ is differentiable with respect to $\theta$, with
\[\nabla_{\theta}\mu_{\theta_0}(t,Z)=\mu '_{\theta_0}(t|\theta_0'Z)\big(Z-E(Z|\theta_0'Z)\big),\]
where, as a consequence
\begin{equation} \label{espcond}
E\big[\nabla_{\theta}\mu_{\theta_0}(t,Z)|\theta_0'Z\big]=0.
\end{equation}
\end{lemma}

\begin{proof}
Observe that $\mu_{\theta}(t,\theta 'Z)=E[\mu_{\theta_0}(t,\theta_0'Z)|\theta 'Z]$
and let $\zeta(Z,\theta)=\theta_0'Z-\theta 'Z$ for $\theta \in \Theta$. We have
\[\mu_{\theta}(t,\theta 'Z)=E\big[\mu_{\theta_0}\big(t,\zeta(Z,\theta)+\theta 'Z\big)|\theta 'Z\big].\] Defining $\Gamma(\theta_1,\theta_2)=E[\mu_{\theta_0}(t,\zeta(Z,\theta_1)+\theta_2'Z)|\theta_2'Z],$ we have $\Gamma(\theta,\theta)= \mu_{\theta}(t,\theta 'Z),$
which leads to
\begin{align*}
\left.\nabla_{\theta_1}\Gamma(\theta_1,\theta_0)\right|_{\theta_1=\theta_0} & =-\mu'_{\theta_0}(t,\theta_0'Z)E\left[Z|\theta_0'Z\right],\\
\left.\nabla_{\theta_2}\Gamma(\theta_0,\theta_2)\right|_{\theta_2=\theta_0} & =Z\mu_{\theta_0}'(t,\theta_0'Z).
\end{align*}
\end{proof}

\section{Auxiliary lemma for tightness conditions}
%\noindent {\bf 6.4.2 Auxiliary lemma for tightness conditions}

\begin{lemma} \label{lemmatightness} 
Let $\mathcal{F}$ be a class of functions.
Let $P_n(t,f)$ be a process on $[0,\tau_H]\times \mathcal{F}$.
Define, for any $\tau\in [0,\tau_H],$ $R_n(\tau,f)=P_n(\tau_H,f)-P_n(\tau,f).$ Assume that for any $\tau$ such that $\tau<\tau_H$
\[P_n(t,f) \Longrightarrow \mathds{W}(V_{f}(t)) \in \mathcal{D}([0,\tau]), f \in \mathcal{F},\]
where $\mathds{W}(V_{f}(t))$ is a centered Gaussian process with covariance function $V_{f}$ and $\mathcal{D}$ denotes the set of càdlàg functions.

\noindent Assume that, for a sequence of random variables $(X_n)$ and two functions $\Gamma$ and $\Gamma_n$, the following conditions hold
\begin{enumerate}[(1)]
\item $\lim_{\tau \rightarrow
\tau_H}V_{f}(\tau)=V_{f}(\tau_H)$ with $\sup_{f \in
\mathcal{F}}|V_{f}(\tau_H)|<\infty,$ 
\item $|R_n(\tau',f)|\leq X_n \times \Gamma_n(\tau)$  for all $\tau<\tau'<\tau_H,$
\item $X_n=O_P(1),$
\item $\Gamma_n(\tau) \to \Gamma(\tau)$ in probability, 
\item $\lim_{\tau \rightarrow\tau_H}\Gamma(\tau)=0.$
\end{enumerate}
Then $P_n(\tau_H,f)\Longrightarrow \mathcal{N}(0,V_{f}(\tau_H)).$
\end{lemma}

\begin{proof}
From Theorem 13.5 in Billingsley~\citeyearpar{Billingsley99} and condition $(1)$, it suffices to show that, for all $\varepsilon>0$
\begin{equation} \label{tension1} \lim_{\tau \rightarrow \tau_H}\limsup_{n
\rightarrow \infty} P\left( \sup_{\tau\leq t\leq \tau_H, f \in
\mathcal{F}}|R_n(t,f)|>\varepsilon\right)=0.
\end{equation}
Using condition $(2)$ , the probability in equation (\ref{tension1}) is bounded, for all $M>0,$
by
\begin{equation}
\label{proba2}
P\big(|\Gamma_n(\tau)-\Gamma(\tau)|>\varepsilon/M-\Gamma(\tau)\big)+P(X_n>M).
\end{equation}
Moreover, from condition $(4)$, we can state that
\[\limsup_{n
\rightarrow \infty} P\big(|\Gamma_n(\tau)-\Gamma(\tau)|>\varepsilon/M-\Gamma(\tau)\big)=I(\varepsilon/M-\Gamma(\tau)\geq 0).\]
Since $\Gamma(\tau)\rightarrow 0$ (condition $(5)$), we can deduce that
\[\lim_{\tau\rightarrow \tau_H} \limsup_{n
\rightarrow \infty} P\big(|\Gamma_n(\tau)-\Gamma(\tau)|>\varepsilon/M-\Gamma(\tau)\big)=0.\]
%Hence,
%\[ \lim_{\tau \rightarrow \tau_H}\limsup_{n
%\rightarrow \infty} P\left( \sup_{\tau\leq t\leq \tau_H, f \in
%\mathcal{F}}|R_n(t,f)|>\varepsilon\right)\leq \limsup_{n \rightarrow \infty}P(Z_n>M).\]
As a consequence,
\[\lim_{\tau \rightarrow \tau_H}\limsup_{n
\rightarrow \infty} P\left( \sup_{\tau\leq t\leq \tau_H, f \in
\mathcal{F}}|R_n(t,f)|>\varepsilon\right)\leq \lim_{M\rightarrow \infty}\limsup_{n \rightarrow \infty}P(X_n>M)=0,\]
using condition $(3)$.
\end{proof}

\section{Covering number results}
%\noindent {\bf 6.4.2 Covering number results}

\label{secclasses}

In this section, we determine the covering numbers of some particular classes of functions. From these computations, sufficient conditions can be deduced  to check Property 2 and Assumption 9.

%For this purpose, define the bracketing number $N(\varepsilon,\mathcal{G},\|\cdot\|_{2})$ of a class of bounded functions $\mathcal{G}$ as the smaller number of $\varepsilon-$brackets needed to cover $\mathcal{G},$ where a $\varepsilon-$bracket is a set $[f,g]=\{h:f\leq h \leq g\}$ with $\|f-g\|_{2}.$ We recall that (see Van der Vaart and Wellner \citeyearpar{Vaart96}) a class $\mathcal{G}$ with finite bracketing number for all $\varepsilon>0$ is Glivenko-Cantelli and is Donsker if $\int_0^1 [\log N_{[]}(\varepsilon,\mathcal{G},\|\cdot\|_{2})]d\varepsilon<\infty.$ Important examples of computations of covering numbers can be found in Van der Vaart and Wellner \citeyearpar{Vaart96}.

\begin{proposition}
Let $\mathcal{F}$ be a class of functions $f(t,z)$ with envelope $\bar F$ defined on $\mathbb{R}\times \mathbb{R}^d$ with continuous derivative with respect to the first component. 
Let $\tilde F$ be the envelope of the class of functions $\partial f(s,z)/\partial s$.
Let $W(t)$ be a positive bounded decreasing function
and set $\mathcal{W}=\{w:dw(t)=W(t)d\tilde{w}(t),\tilde w \in \tilde{\mathcal{W}}\}$ where $\tilde{\mathcal{W}}$ is a class of monotone positive functions with envelope function $\tilde W$.

\noindent Assume that $E[(\int_0^{\tau_H}\bar F(t,z)W(t)dY(t))^2]<\infty$, $E[(\int_0^{\tau_H}\bar F(t,z)Y(t)dW(t))^2]<\infty$
and\\$E[(\int_0^{\tau_H}\tilde F(t,z)W(t)Y(t)dt)^2]<\infty.$

\noindent Then, the class of functions $\mathcal{H}=\{(z,y)\rightarrow \int_0^{\tau_H}f(t,z)y(t)dw(t),f\in \mathcal{F},w\in \mathcal{W}\}$ has a uniform covering number satisfying, for some constant $C$,
\[N(\varepsilon,\mathcal{H},\|\cdot\|_2)\leq C N(\varepsilon,W\mathcal{F},\|\cdot\|_2)N(\varepsilon,\tilde{\mathcal{W}},\|\cdot\|_2).\]

\end{proposition}

\begin{proof}
Let $Q$ be a probability measure and introduce $N_1=\sup_QN(\varepsilon\|W\bar F\|_Q,W\mathcal{F},\|\cdot\|_{2,Q})$ and $N_2=\sup_QN(\varepsilon\|\tilde W\|_Q,\tilde{\mathcal{W}},\|\cdot\|_{2,Q}).$
Let $\{f^{W}_{i}\}_{1\leq i\leq N_1}$ (respectively $\{\tilde{w}_{j}\}_{1\leq j\leq N_2}$) be the center of the $\varepsilon-\|\cdot\|_{2,Q}$ balls needed to cover $W\mathcal{F}$ (respectively $\tilde{\mathcal{W}}$). 
Writing $dw(t)=W(t)d\tilde{w}(t)$, we have for any ${1\leq i\leq N_1}$ and ${1\leq j\leq N_2}$ 
\begin{align*}
 & \left|\int_0^{\tau_H} Y(t)f(t,z) W(t)d\tilde w(t)-\int_0^{\tau_H} Y(t) f^{W}_{i}(t,z)d\tilde{w}_{j}(t)\right|\\
 & \quad \leq \left|\int_0^{\tau_H} Y(t)\big(f(t,z)W(t)-f^{W}_{i}(t,z)\big)d\tilde{w}_{j}(t)\right|+\left|\int_0^{\tau_H} Y(t)f(t,z) W(t)(d\tilde{w}-d\tilde{w}_{j})(t)\right|.
\end{align*}
For any $f\in \mathcal{F},$ there exists some $i$ such that the first term
is seen to be less than $c_1\varepsilon$ in $L^2(Q)-$norm. For the second term, there also exists some $j$ such that this is smaller than $c_2\varepsilon,$ which can be seen using integration by parts. The result follows.
\end{proof}

\bibliographystyle{chicago}
\bibliography{BGLbibli}